\documentclass[review,hidelinks,onefignum,onetabnum]{siamart251216}

\usepackage{amsfonts}
\usepackage{mathrsfs}
\usepackage{bm}
\usepackage{tikz}
\usetikzlibrary{positioning, shapes.geometric}
\usetikzlibrary{angles,quotes,calc}
\usepackage{graphicx}
\usepackage{epstopdf}
\usepackage{xcolor}

\newsiamthm{assumption}{Assumption}
\newsiamthm{hypothesis}{Hypothesis}
\newsiamthm{example}{Example}
\newsiamremark{remark}{Remark}

\definecolor{revisiongreen}{RGB}{0,128,0}

\definecolor{reviewblue}{RGB}{0,70,190}

\definecolor{reviewred}{RGB}{190,0,0}

\newcommand{\ud}{\mathrm d}
\newcommand{\R}{\mathbb{R}}
\numberwithin{equation}{section}

\newcommand{\E}{\mathbb{E}}\allowdisplaybreaks[4]

\headers{Asymptotic preserving method}{J. Cui, G. Li, and D. Sheng}

\title{Asymptotic preservation of an exponential Euler method under Hamiltonian fast advection with multiple critical points}

\author{Jianbo Cui\thanks{Department of Applied Mathematics, The Hong Kong Polytechnic University, Hung Hom, Kowloon, Hong Kong (\email{jianbo.cui@polyu.edu.hk, guozhen.li@polyu.edu.hk}).}
\and Guozhen Li\footnotemark[1]
\and Derui Sheng\thanks{Department of Applied Mathematics, The Hong Kong Polytechnic University, Hung Hom, Kowloon, Hong Kong (\email{sdr@lsec.cc.ac.cn}).}}

\ifpdf
\hypersetup{
  pdftitle={Asymptotic preservation of an exponential Euler method under Hamiltonian fast advection with multiple critical points},
  pdfauthor={Jianbo Cui, Guozhen Li, and Derui Sheng}
}
\fi

\AtBeginDocument{}

\begin{document}

\maketitle

\begin{abstract}
We establish the asymptotic-preserving property of the exponential Euler method for stochastic reaction-diffusion-advection equations in $\mathbb R^2$ under fast Hamiltonian advection with multiple critical points. In the fast-advection limit, the dynamics reduces to a stochastic partial differential equation on a noncompact metric graph. A key ingredient is the strong convergence analysis of the limiting exponential Euler scheme, which is complicated by the graph's noncompactness, its multiple edges and vertices, and the nonuniform ellipticity and vertex degeneracy of the graph operator. Using a weighted $L^2$ framework and analytic-semigroup smoothing estimates, we prove a temporal convergence order arbitrarily close to $1/2$ for weighted $L^2$ initial data and exactly $1/2$ under a half-order fractional-domain condition. Combining these estimates with the fast-advection limits yields the asymptotic-preserving property in the multiple-critical-point setting. Numerical experiments illustrate the asymptotic behavior and confirm the predicted temporal convergence rate.
\end{abstract}


\section{Introduction}
Consider an incompressible flow in $\mathbb R^2$ with stream function
$-H:\mathbb R^2\to\mathbb R$. The density of particles satisfies the Liouville
equation in the inviscid regime, and is transported along the level sets of $H$. In
certain cases, the flow has a small viscosity and the particles undergo a slow
chemical reaction with both deterministic and stochastic components, so the
motion of the particles is described by the following stochastic reaction--diffusion--advection (RDA) equation \cite{CF19}
\begin{equation*}
\left\{
\begin{aligned}
\partial_t\widetilde u_\epsilon(t,x)
&=\left(\frac{\epsilon}{2}\Delta
+\nabla^\perp H(x)\cdot\nabla\right)\widetilde u_\epsilon(t,x)
+\epsilon b\bigl(\widetilde u_\epsilon(t,x)\bigr)+\sqrt{\epsilon}\,g\bigl(\widetilde u_\epsilon(t,x)\bigr)
\partial_t\widetilde{\mathcal W}(t,x),\\
\widetilde u_\epsilon(0,x)&=\chi(x),\quad (t,x)\in(0,T]\times \R^2.
\end{aligned}
\right.
\end{equation*}
Here, $\nabla^\perp H=(-\partial_{x_2}H,\partial_{x_1}H)$, $\widetilde{\mathcal W}$ is a spatially homogeneous Wiener process, $\chi:\mathbb R^2\to\mathbb R$ is the non-random initial value, and
$\epsilon>0$ is a small parameter. Moreover, the Hamiltonian
$H:\mathbb R^2\to[0,\infty)$ grows quadratically at infinity, and
we assume that $b,g:\mathbb R\to\mathbb R$ are Lipschitz continuous; see section~\ref{S:Pr} for more details. To
describe the long-time behavior of $\tilde{u}_\epsilon$ over the time period $[0,T/\epsilon]$, we set
$u_\epsilon(t,x):=\widetilde u_\epsilon(t/\epsilon,x)$ and obtain the following stochastic RDA equation with fast Hamiltonian advection:
\begin{equation}\label{eq:SRDA-intro}
\left\{
\begin{aligned}
	\partial_t u_\epsilon(t,x)
	&=\mathcal L_\epsilon u_\epsilon(t,x)
	\!+ b(u_\epsilon(t,x))
	+ g(u_\epsilon(t,x))\partial_t\mathcal{W}(t,x),\\
	u_\epsilon(0,x)&=\chi(x),\quad (t,x)\in (0,T]\times \R^2,
\end{aligned}
\right.
\end{equation}
where $\mathcal W(t,x):=\sqrt{\epsilon}\,\widetilde{\mathcal W}(t/\epsilon,x)$. The operator
$
\mathcal L_\epsilon
:=\frac{1}{2}\Delta+\frac{1}\epsilon\nabla^\perp H\cdot\nabla
$
is the infinitesimal generator of
the diffusion process:
\begin{equation}\label{eq:Xintro}
	\ud X_\epsilon(t)
	=
	\frac{1}\epsilon\nabla^\perp H\big(X_\epsilon(t)\big)\ud t
	+\ud \textup{B}(t),\quad X_\epsilon(0)=x\in \R^2,
\end{equation}
where $\textup{B}$ is a two-dimensional Brownian motion. 

 The multiscale model \eqref{eq:SRDA-intro} reduces to a stochastic partial differential equation (SPDE) on a {\color{black}metric graph} $\Gamma$ in the fast-advection limit. Indeed,
 collapsing each component of a level set
\(\mathrm C(z):=\{x\in\mathbb R^2:H(x)=z\}\) to a point 
yields a noncompact {\color{black}metric graph} \(\Gamma\) with finitely many edges and
vertices, the latter corresponding to the critical points of \(H\);
see section~\ref{S:2.1}. We denote by
\(\Pi:\mathbb R^2\to\Gamma\) the associated projection,
$
    \Pi(x)=(H(x),k(x)),
$
where \(k(x)\) labels the connected component of the level set
containing \(x\). By an averaging principle,
\(\Pi(X_\epsilon)\) converges in distribution to a Markov process
\(\bar Y\) on \(\Gamma\) \cite{FW12}. Moreover, by introducing a
suitable graph weight \(\gamma:\Gamma\to(0,\infty)\), it was shown
in \cite{CF19} that, for every \(\tau_0\in(0,T)\),
\begin{equation}\label{eq:u-u}
\lim_{\epsilon\to0}
\E\bigg[
\sup_{t\in[\tau_0,T]}
\|u_\epsilon(t)-u(t)\circ\Pi\|_{\mathbb H_\gamma}^2
\bigg]
=0\quad \text{with } \mathbb H_\gamma
:=L^2(\mathbb R^2;\gamma(\Pi(x))\ud x).
\end{equation}
The limiting process $\{u(t)\}_{t\in[0,T]}$ in \eqref{eq:u-u} is governed by the stochastic PDE
\begin{equation}\label{eq:intro}
\left\{
\begin{aligned}
	\partial_t u(t,z,k)
	&=\mathcal Lu(t,z,k)
	+b\big(u(t,z,k)\big)
	+g\big(u(t,z,k)\big)\partial_tW(t,z,k),\\
	u(0,z,k)&=\chi^\wedge(z,k),
\end{aligned}
\right.
\end{equation}
for $t\in(0,T]$ and $(z,k)\in\Gamma$. The Wiener process $W$ and the
initial value $\chi^\wedge$ are the formal projections of
$\mathcal W$ and $\chi$ onto the {\color{black}metric graph} $\Gamma$, respectively; see
section~\ref{S3}. The operator $\mathcal L$ is the infinitesimal
generator of the limiting process $\bar Y$. It acts as a second-order
differential operator on each edge $I_k$, degenerates near the vertices $\{O_i\}$,
and is endowed with Kirchhoff-type vertex gluing conditions; see \eqref{eq:glue}.

Equation \eqref{eq:intro} belongs to the broader class of stochastic
PDEs on metric graphs (see, e.g., \cite{CF19,CHW24}), which arise in a variety of applications.
 In neuroscience, for instance, stochastic PDEs on {\color{black}metric graphs} can describe
the propagation of electrical signals along dendritic trees, with
stochastic impulsive inputs accounting for excitatory and inhibitory
influences from neighboring neurons
\cite{BMZ08,BM10}. Beyond neuroscience, such equations arise
in free-electron models for organic molecules \cite{LP36},
superconductivity in granular and engineered materials \cite{AS83},
wave propagation in acoustic and electromagnetic networks
\cite{CRH87}, Anderson transitions in disordered wires
\cite{AMR20,SB82}, quantum chaos \cite{BK13}, and statistical
modeling \cite{BSW24}.
Since \eqref{eq:SRDA-intro} generally has no explicit solution, numerical discretization provides a practical tool for investigating the long-time dynamics of particles in a stochastic incompressible viscous fluid. A key numerical difficulty is to preserve the correct limiting dynamics
\eqref{eq:intro} in the fast-advection regime, which motivates the notion of asymptotic preserving (AP). A time discretization of \eqref{eq:SRDA-intro} is called AP if, for every fixed time step $\tau=T/N$, its
limit as $\epsilon\to0$ is a consistent time discretization of the limiting
graph equation \eqref{eq:intro}. Once the two one-parameter limits are known
to exist, the AP property is equivalently expressed by commutation of the
limits $\epsilon\to0$ and $N\to\infty$.
As emphasized in \cite{Jin12}, when the small scale is not
numerically resolved, an AP scheme of \eqref{eq:SRDA-intro} automatically becomes a macroscopic solver
for the limiting equation \eqref{eq:intro}.

\begin{figure}[htbp]
\centering
\begin{tikzpicture}[
    >=stealth,
    state/.style={
        font=\normalsize,
        inner sep=4pt
    },
    reference/.style={
        font=\scriptsize
    }
]
\node[state] (UL) at (0,0)      {$U_\epsilon^N$};
\node[state] (UR) at (7.0,0)    {$U^N\circ\Pi$};
\node[state] (LL) at (0,-2.0)   {$u_\epsilon(T)$};
\node[state] (LR) at (7.0,-2.0) {$u(T)\circ\Pi$};

\draw[->,thick]
    (UL)--(UR)
    node[midway,above=3pt] {$\epsilon\to0$}
    node[midway,below=3pt,reference]
    {\cite[Theorem~2]{CS24}};

\draw[->,thick]
    (LL)--(LR)
    node[midway,above=3pt] {$\epsilon\to0$}
    node[midway,below=3pt,reference]
    {\cite[Theorem~5.3]{CF19}};

\draw[->,thick]
    (UL)--(LL)
    node[midway,left=6pt] {$N\to\infty$}
    node[midway,right=6pt,reference]
    {\cite[Theorem~1]{CS24}};

\draw[->,thick]
    (UR)--(LR)
    node[midway,right=6pt] {$N\to\infty$}
    node[midway,left=6pt,reference]
    {Theorem~\ref{thm:main2}};
\end{tikzpicture}
\caption{Commutative diagram illustrating the AP
property.}
\label{fig.doubleDiagram}
\end{figure}

In this work, we study the exponential Euler scheme for \eqref{eq:SRDA-intro}, with a particular focus on its AP property.
 Let $U_\epsilon^N$ denote the corresponding
numerical approximation of $u_\epsilon(T)$ with step size
$\tau=T/N$. It was proved in  \cite{CS24} that, for every fixed $\epsilon>0$,
$U_\epsilon^N$ converges to $u_\epsilon(T)$ as $N\to\infty$, while,
for every fixed $N$, $U_\epsilon^N$ converges as $\epsilon\to0$ to
the lift of the numerical solution $U^N$ of the limiting equation \eqref{eq:intro}.
More precisely,
\begin{equation}\label{eq:UU}
\lim_{N\to\infty}U_\epsilon^N=u_\epsilon(T),
\qquad
\lim_{\epsilon\to0}U_\epsilon^N=U^N\circ\Pi
\quad\text{in }L^2(\Omega,\mathbb H_\gamma).
\end{equation}
 Accordingly, the AP property in the present
setting means that
\begin{equation}\label{eq:AP-intro}
\lim_{N\to\infty}\lim_{\epsilon\to0}U_\epsilon^N
=
u(T)\circ\Pi
=
\lim_{\epsilon\to0}\lim_{N\to\infty}U_\epsilon^N
\quad\text{in }L^2(\Omega,\mathbb H_\gamma).
\end{equation}
In view of \eqref{eq:u-u} and \eqref{eq:UU}, the analysis of \eqref{eq:AP-intro} reduces to establishing the strong convergence
of $U^N$ to $u(T)$; see
Fig.~\ref{fig.doubleDiagram}. 
In fact, the sequence $\{U^n\}_{n=0}^N$ is precisely the exponential Euler
approximation of the limiting equation \eqref{eq:intro}, given
by
\begin{equation}\label{eq:Un-intro}
U^n
=
e^{\tau\mathcal L}U^{n-1}
+\tau e^{\tau\mathcal L}B(U^{n-1})
+e^{\tau\mathcal L}G(U^{n-1})\delta W_n,
\qquad n=1,\ldots,N,
\end{equation}
with $U^0=\chi^\wedge$, where
$\delta W_n:=W(t_n)-W(t_{n-1})$, $t_n:=n\tau$, and $B$ and $G$ are
the Nemytskii operators associated with $b$ and $g$, respectively;
see \eqref{eq:BG}.
In \cite{CS24}, the strong convergence of $U^N$
was established under the assumption that the Hamiltonian $H$ has a
unique critical point, in which case the associated {\color{black}metric graph} $\Gamma$
reduces to the half-line $[0,\infty)$, consisting of a single edge with no
branching or interior vertices. Hence, the existing analysis 
does not address the genuinely {\color{black}multi-edge metric-graph} setting arising
from Hamiltonians with multiple critical points, which motivates the present work.

Compared to \cite{CS24}, we establish the strong
convergence of the limiting scheme \eqref{eq:Un-intro} on a noncompact,
{\color{black}multi-edge metric graph} with interior vertices.  This extension is not formal because
the limiting generator $\mathcal{L}$ is degenerate at the vertices and its domain couples all
incident edges through Kirchhoff conditions. %
 More generally,
the convergence analysis of numerical approximations for stochastic
PDEs on {\color{black}metric graphs} remains at an early stage
\cite{BKKS24,CKS26}, mainly focused on spatial discretizations. In
\cite{BKKS24}, a method combining finite elements with rational
approximations of fractional operators was proposed for stochastic
fractional elliptic equations on compact metric graphs. More recently,
\cite{CKS26} developed a spatial discretization of
\eqref{eq:intro} based on localization, regularization, and finite
element techniques. For the temporal approximation
\eqref{eq:Un-intro},
the strong convergence analysis presents several difficulties. First, the Kirchhoff-type conditions globally couple the operator domain across different edges, complicating the regularity estimates of \eqref{eq:intro} required in the error analysis. Second, the operator
$\mathcal L$ is not uniformly elliptic and degenerates near the
vertices, which makes semigroup estimates more delicate. Finally,
the driving process $W$ generally lacks $H^1(\Gamma)$-regularity because
its spatial derivatives exhibit singular behavior near the vertices.
Consequently, standard arguments that treat $W$ as an
$H^1(\Gamma)$-valued Wiener process are not applicable.

To overcome these difficulties, 
the error analysis of \eqref{eq:Un-intro} is carried out in the weighted space
$L^2_{\beta\gamma}(\Gamma)$, where
$\beta_k(z)$ is the period of the Hamiltonian
motion along $\mathrm C_k(z)$. In this space, the operator
$\mathcal L$ generates an analytic semigroup \cite{CKS26}, whose
smoothing properties remain available despite the degeneracy of
$\mathcal L$ near the vertices. A central tool is the square-function
estimate
\begin{equation*}%
\int_0^T
\big\|
(\lambda I-\mathcal L)^{\frac{1}{2}}e^{t\mathcal L}\phi
\big\|_{L^2_{\beta\gamma}(\Gamma)}^2
\,\ud t
\leq
C\|\phi\|_{L^2_{\beta\gamma}(\Gamma)}^2,
\qquad
\phi\in L^2_{\beta\gamma}(\Gamma),
\end{equation*}
for some $\lambda,C>0$. It controls the half-order fractional derivative
of the stochastic convolution and avoids the stronger, generally false,
requirement that the driving Wiener process be $H^1(\Gamma)$-valued. Together with the temporal
H\"older regularity of the mild solution, it yields the strong
convergence estimates for the exponential Euler approximation.
The AP property then follows by combining this convergence with the
continuous fast-advection limit \eqref{eq:u-u} and the discrete
fast-advection limit \eqref{eq:UU}.

Our main contributions are summarized as follows.

\begin{itemize}

\item[(i)] We establish the optimal temporal H\"older regularity estimate of the mild solution to the stochastic PDE \eqref{eq:intro} posed on a noncompact {\color{black}multi-edge metric graph}; see Theorem~\ref{lem:opt-Hol}.
\item[(ii)] We prove strong convergence of every order below $1/2$ for weighted-$L^2$ initial data and of order $1/2$ under a half-order fractional-domain condition for the exponential Euler approximation \eqref{eq:Un-intro}; see Theorem~\ref{thm:main2}.  To the best of our knowledge, this is the
first strong convergence result for a temporal discretization of a
stochastic PDE on a noncompact {\color{black}multi-edge metric graph}.

\item[(iii)] Based on the strong convergence result, we establish the AP
property \eqref{eq:AP-intro} when the Hamiltonian has multiple
critical points; see Theorem~\ref{coro:AP}.

\end{itemize}

The remainder of this paper is organized as follows.
Section~\ref{S:Pr} introduces the Hamiltonian fast-advection
framework and the associated {\color{black}metric graph}, and presents the continuous and
discrete fast-advection limits. Section~\ref{S3} studies the
stochastic PDE \eqref{eq:intro} and establishes the regularity estimates needed for
the numerical analysis. Section~\ref{S4} proves the strong convergence
of the exponential Euler approximation \eqref{eq:Un-intro} and derives the AP property \eqref{eq:AP-intro}.
Section~\ref{S:NE} presents a fully discrete implementation
of \eqref{eq:intro} and the numerical
experiments.%

\section{Hamiltonian fast-advection limits and exponential Euler discretization}\label{S:Pr}
This section introduces the graph structure and the associated
weighted spaces, and recalls the continuous fast-advection limit of
\eqref{eq:SRDA-intro} and the discrete fast-advection limit of its
exponential Euler approximation.

We first introduce some notation. In the sequel, we denote by $C$ a generic constant independent of the time step size $\tau:=T/N$, whose specific value may vary between occurrences. When
necessary, we write $C(a,b,\ldots)$ to emphasize its dependence upon
parameters $a,b,\ldots$.
Given two Hilbert spaces $V_1$ and $V_2$, we denote by $\mathscr{L}(V_1)$ the space of bounded linear operators from $V_1$ to itself, equipped with the operator norm, 
 and by $\mathscr{L}_2(V_1, V_2)$ the space of Hilbert--Schmidt operators from $V_1$ to $V_2$, endowed with the Hilbert--Schmidt norm
(cf.~\cite{DZ14}). Henceforth, for a Banach space $V$, $\mathcal{C}([0,T];V)$ denotes the space of continuous functions from $[0,T]$ to $V$ endowed with the supremum norm. Similarly, for any $p\in[1,\infty)$, $L^p(\Omega,V)$ denotes the space of $p$-integrable random variables $Z:\Omega\to V$ equipped with the norm $\|Z\|_{L^p(\Omega,V)}:=(\E[\|Z(\omega)\|_{V}^p])^{\frac{1}{p}}$.

\subsection{Hamiltonian, projection, and {\color{black}metric graph}}\label{S:2.1}
Recall that $-H$ is the stream function associated with the
incompressible velocity field $\nabla^\perp H$, and that
$\mathcal L_\epsilon$ is the infinitesimal generator of the diffusion
process $X_\epsilon$ given by \eqref{eq:Xintro}.
Throughout this paper, assume that the Hamiltonian $H$ satisfies the following conditions \cite[Hypothesis 1]{CF19}:
\begin{enumerate}
	\item[(i)] $ H \in \mathcal{C}^4(\R^2)$ with bounded second derivatives and
	$\min _{x \in \mathbb{R}^2}  H (x)=0;
	$
	\item[(ii)] $ H $ has finitely many critical points $ \mathbf{x}_1, \ldots, \mathbf{x}_{m_1}$. 
	The Hessian matrix $\nabla^2  H (\mathbf{x}_i)$ is non-degenerate for every $i=1,2, \ldots, {m_1}$, and $ H (\mathbf{x}_i) \neq H (\mathbf{x}_j)$ if $i \neq j$;
	\item[(iii)] There exist $\mathfrak{a}_1, \mathfrak{a}_2, \mathfrak{a}_3>0$ such that
	for all $x \in \mathbb{R}^2$ with $|x|$ large enough,
	\begin{equation*}
		H (x) \ge \mathfrak{a}_1|x|^2,\quad|\nabla  H (x)| \ge \mathfrak{a}_2|x|,\quad\Delta  H (x) \ge \mathfrak{a}_3.
	\end{equation*}
\end{enumerate}

For each connected component $\mathrm{C}_k(z)$ of the level set 
$
\mathrm{C}(z):=\{x\in\mathbb{R}^2:  H(x)=z\},
$
the probability measure $\mu_{z,k}$, given by
\begin{equation}\label{eq:mubeta}
\mathrm{d}\mu_{z,k}(x)=\frac{1}{\beta_k(z)} |\nabla H(x)|^{-1} \mathrm{d}l_{k,z},\qquad \beta_k(z):=\oint_{\mathrm{C}_k(z)}|\nabla H(x)|^{-1} \mathrm{d}l_{k,z}
\end{equation}
is invariant for the Hamiltonian system $\dot{X}(t)=\nabla^\perp H(X(t))$, where $\ud l_{z,k}$ is the length element on $\textup{C}_k(z)$.
By collapsing all points in $\mathbb{R}^2$ that lie in the same connected component $\mathrm{C}_k(z)$ into a single point, one obtains a {\color{black}metric graph} $\Gamma$ composed of finitely many vertices $\{O_i\}_{i=1}^{m_1}$ and edges $\{I_k\}_{k=1}^m$. Each vertex $O_i$ is associated with a critical point $\mathbf{x}_i$ of $H$: interior and exterior vertices correspond to saddle points and local extrema of $H$, respectively.
By a harmless abuse of notation, we write $H(O_i):=H(\mathbf{x}_i)$ for the critical energy associated with $O_i$. 
In addition, we regard $O_{\infty}$ as an exterior vertex, corresponding to the endpoint of the unique unbounded edge $I_m$, which represents the point at infinity (see Fig.~\ref{Hamilton_level} for an illustration). Denote by $\Pi:\mathbb{R}^2\to\Gamma$ the projection map that assigns to each $x\in\mathbb{R}^2$ its equivalence class on $\Gamma$. More precisely,
$
\Pi(x) = (H(x),  k(x)),
$
where $k(x)\in\{1,\ldots,m\}$ indicates the edge $I_{k(x)}$ containing $\Pi(x)$. In this way, each edge $I_k$ connecting two vertices $O_i$ and $O_j$ can be identified with the interval $[H(\mathbf{x}_i), H(\mathbf{x}_j)]$, where $H(\mathbf{x}_i) < H(\mathbf{x}_j)$, and we write $I_k \cong [H(\mathbf{x}_i), H(\mathbf{x}_j)]$. We also denote by $\mathring{I}_k\cong (H(\mathbf{x}_i), H(\mathbf{x}_j))$ the interior of the edge $I_k$. In addition, $I_m \cong [H_0, \infty)$, where
$
H_0 := \max_{1 \le i \le m_1} H(\mathbf{x}_i).$ 
A function $f$ defined on the {\color{black}metric graph} $\Gamma$ is said to be continuous at a vertex $O_i$ if for all edges $\{I_k\}$ incident to $O_i$, the corresponding coordinate functions $\{f(\cdot,k)\}$ attain the same value at $O_i$. Furthermore, a function $f$ is said to be continuous on the graph $\Gamma$ if it is continuous on every edge and continuous at every vertex. By averaging with respect to the invariant measure $\mu_{z,k}$, it was proved in \cite[Chapter 8]{FW12} that for any deterministic initial condition $X_\epsilon(0)\in\mathbb{R}^2$, the projected process $\{\Pi(X_\epsilon(t))\}_{t\in[0,T]}$ converges in distribution to a Markov process $\{\bar{Y}(t)\}_{t\in[0,T]}$ on the graph $\Gamma$, described by its infinitesimal generator $\mathcal{L}$ (see section \ref{S3.1} for more details).

\begin{figure}[!htb]
 \centering
 \includegraphics[width=0.96\linewidth]{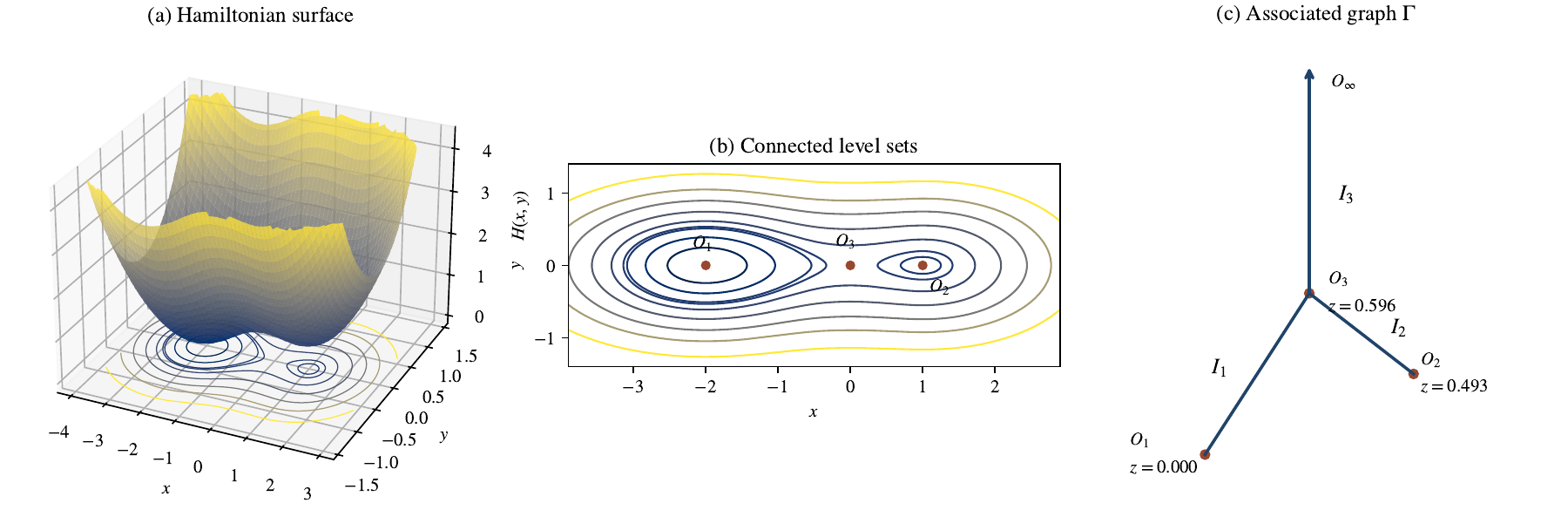}
 \caption{(a) The Hamiltonian surface and its projected level curves.
 (b) Level sets in the $(\mathtt{x},\mathtt{y})$-plane.
 (c) The associated {\color{black}metric graph} $\Gamma$.}
 \label{Hamilton_level}
\end{figure}

\subsection{Fast advection asymptotics for (1.1)}
Let the driving process $\mathcal{W}$ of \eqref{eq:SRDA-intro} be a spatially homogeneous Wiener process on $\mathbb{R}^2$ with a finite nonnegative-definite symmetric spectral measure $\mu$. Namely, $\mathcal{W}$ is a Gaussian
random field on $[0,\infty)\times\mathbb{R}^2$ defined on the stochastic basis
$(\Omega,\mathscr{F},\{\mathscr{F}_t\}_{t\geq 0},\mathbb{P})$ such that for every
$t,s\geq 0$ and $x,y\in\mathbb{R}^2$,
\begin{equation*}%
\mathbb{E}\big[\mathcal{W}(t,x)\mathcal{W}(s,y)\big]
= (t\wedge s)\Lambda(x-y),
\end{equation*}
where $\Lambda$ is the Fourier transform of $\mu$, i.e.,
$
\Lambda(x)=\int_{\mathbb{R}^2} e^{\mathbf{i} x\cdot \xi}\,\mu(\ud\xi)
$
for $x\in\mathbb{R}^2$. Here $\mathbf{i}$ is the imaginary unit.
Denote by $L^2_{(s)}(\R^2,\ud \mu)$ the subspace of the Hilbert space $L^2(\R^2,\ud \mu;\mathbb{C})$ consisting of all functions $\varphi$ satisfying $ \overline{\varphi(-x)}=\varphi(x)$, $x \in\R^2$.  By \cite[Proposition 1.2]{PZ97}, an orthonormal basis for the reproducing kernel space of $\mathcal{W}$  is given by $\{\widehat{{\mathfrak{u}_j\mu}}\}_{j\in\mathbb{N}_+}$, where $\{{\mathfrak{u}_j}\}_{j\in\mathbb{N}_+}$ is a complete orthonormal basis of the Hilbert space $L^2_{(s)}(\R^2,\ud \mu)$ and
$\widehat{{\mathfrak{u}_j\mu}}(x)=\int_{\R^2}e^{\mathbf{i}x\cdot\xi} {\mathfrak{u}_j}(\xi)\mu(\ud\xi)$ for $x\in\R^2$.
Consequently, $\mathcal W$ admits the following Karhunen--Lo\`eve expansion
\begin{equation*}%
	\mathcal W(t,x)=\sum_{j=1}^\infty\widehat{{\mathfrak{u}_j\mu}}(x)\tilde{\boldsymbol{\beta}}_j(t),\quad t \geq 0,~x\in\R^2,
\end{equation*}
where $\{\tilde{\boldsymbol{\beta}}_j\}_{j\in\mathbb N_+}$ is a sequence of independent Brownian motions defined on the stochastic basis $(\Omega,\mathscr F,\{\mathscr F_t\}_{t\ge0},\mathbb P)$. By the assumed finiteness of the spectral measure $\mu$, for every $x\in\R^2$ (see \cite{CS24}),
	\begin{equation}\label{eq:ujmu1}
	\sum_{j=1}^\infty|\widehat{{\mathfrak{u}_j\mu}}(x)|^2
	\le \mu(\R^2)<\infty.	
	\end{equation}

In the sequel,
 let $\gamma: \Gamma \to (0,+\infty)$, $(z,k)\mapsto \gamma_k(z)$ be a bounded continuous function satisfying
\begin{equation}\label{eq:IkTk}
\sum_{k=1}^m \int_{I_k} \gamma_k(z)\,\beta_k(z)\, \mathrm{d}z < \infty
\end{equation}
with $\beta_k$ the Hamiltonian period defined in \eqref{eq:mubeta}.
The following assumption ensures that the semigroup $\{e^{t\mathcal{L}_\epsilon}\}_{t\ge 0}$ is uniformly bounded in $\mathscr{L}(\mathbb{H}_\gamma)$, where $\mathbb{H}_\gamma = L^2(\mathbb{R}^2; \gamma(\Pi(x))\, \mathrm{d}x)$, that is,  there exists $C:=C(T)>0$ such that (\cite[Lemma 4.1]{CS24})  
\begin{equation*}
\|e^{t\mathcal{L}_\epsilon}\|_{\mathscr{L}(\mathbb{H}_\gamma)}\le C,\quad \forall~\epsilon>0,t\in[0,T].
\end{equation*}

 \begin{assumption}%
 	\label{Asp:gamma}
 Let $\gamma(z,k)=\vartheta(z)$ for every $(z,k)\in\Gamma$, where
$\vartheta:[0,\infty)\to(0,\infty)$ is twice differentiable satisfying 
$|z\vartheta^{\prime\prime}(z)|+|\vartheta^{\prime}(z)|\le C\vartheta(z)$ for all $z\in[0,\infty)$.
 \end{assumption}

 For a function $\varphi:\mathbb{R}^2\to \mathbb{R}$, we define its projection $\varphi^\wedge:\Gamma\to\mathbb{R}$ by $ \varphi^\wedge(z,k)=\oint_{\mathrm{C}_k(z)} \varphi(x)\,\mathrm{d}\mu_{z,k}, (z,k)\in \Gamma,$ while for
a function $f:\Gamma\to \mathbb{R}$, we denote its lift by $f^\vee(x):=f(\Pi(x)), x\in\R^2.$ The space $\mathbb{H}_\gamma$ projects onto the graph weighted space
\begin{equation*}
L^2_{\beta\gamma}(\Gamma):=\bigg\{f:\Gamma\to\R:\|f\|_{L^2_{\beta\gamma}(\Gamma)}^2=\sum_{k=1}^m\int_{I_k}|f(z,k)|^2\beta_k(z)\gamma_k(z)\ud z<\infty\bigg\}
\end{equation*}
such that 
 \begin{equation}\label{eq:contraction}
 	\|f^\vee\|_{\mathbb{H}_\gamma}=\|f\|_{L^2_{\beta\gamma}(\Gamma)},\quad \|\varphi^\wedge\|_{L^2_{\beta\gamma}(\Gamma)}\le \|\varphi\|_{\mathbb{H}_\gamma},\quad \forall~ f\in L^2_{\beta\gamma}(\Gamma),~\varphi\in \mathbb{H}_\gamma.
 \end{equation}
 It was shown in \cite{CF19} that the limiting process of $\{(u_\epsilon)^\wedge\}_{\epsilon>0}$ is governed by the stochastic PDE \eqref{eq:intro} on the {\color{black}metric graph} $\Gamma$ (see Proposition \ref{tho:faa} below).
The driving process $W$ of \eqref{eq:intro} is the projection of $\mathcal{W}$ given by 
\begin{equation*}
W(t, z, k)=\sum_{j=1}^{\infty}(\widehat{{\mathfrak{u}}_j \mu})^{\wedge}(z, k) \tilde{\boldsymbol{\beta}}_j(t), \quad t \geq 0 ,\quad(z, k) \in \Gamma.
\end{equation*}
Due to \eqref{eq:ujmu1} and the Cauchy--Schwarz inequality, we also have that for any $(z,k)\in\Gamma$,
	\begin{equation}\label{eq:ujmu2}
	\sum_{j=1}^\infty|(\widehat{{\mathfrak{u}_j\mu}})^\wedge(z,k)|^2
	\le \mu(\R^2)<\infty.	
	\end{equation}
	In the sequel, we assume that $\chi\in \mathbb{H}_\gamma$.
	Then under Assumption \ref{Asp:gamma}, both the multiscale model \eqref{eq:SRDA-intro} and its limiting equation \eqref{eq:intro} admit a unique mild solution (see \cite{CS24}).
Moreover, we have the following fast-advection asymptotics of the stochastic RDA equation \eqref{eq:SRDA-intro}.
\begin{proposition}[{\cite[Theorem 5.3]{CF19}}]\label{tho:faa}
Under the standing assumptions on $H$, $b$, $g$, $\mathcal W$, and $\chi$
above, let Assumption \ref{Asp:gamma} hold and suppose that
\begin{equation}\label{eq:Tneq0}
\beta_k'(z)\neq 0,\quad (z,k)\in \mathring{I}_k,\quad k=1,2,\ldots,m.
\end{equation}
 Then for every $p\ge1$ and $\tau_0\in(0,T )$, 
 \begin{equation*}%
\lim_{\epsilon\to0}\E\bigg[\sup_{t\in[\tau_0,T ]}\|u_\epsilon(t)-u(t)^\vee\|_{\mathbb{H}_\gamma}^p\bigg]
=0.
\end{equation*}
\end{proposition}

\subsection{Exponential Euler approximation for (1.1)}\label{S:discretelimit}
In this subsection, we recall the exponential Euler approximation of
\eqref{eq:SRDA-intro} and its fast-advection limit obtained
in \cite{CS24}.
Let $0=t_0<t_1<\cdots<t_N=T$ be a uniform partition of $[0,T]$ with step size $\tau=T/N$, where $N$ is a sufficiently large positive integer.
Denote by $\{U^n_\epsilon\}_{n=0}^N$ the exponential Euler approximation of $\{u_\epsilon(t_n)\}_{n=0}^N$ generated via
\begin{equation}\label{eq:EEMmult}
U^n_\epsilon=e^{\tau\mathcal{L}_\epsilon}U^{n-1}_\epsilon+\tau e^{\tau\mathcal{L}_\epsilon}\tilde{B}(U^{n-1}_\epsilon)
	+e^{\tau\mathcal{L}_\epsilon}\tilde{G}(U^{n-1}_\epsilon)(\mathcal{W}(t_{n})-\mathcal{W}(t_{n-1}))
	\end{equation}
	for $n=1,2,\ldots,N$,
where $\tilde{B}:\mathbb{H}_\gamma\to \mathbb{H}_\gamma$ and $\tilde G:\mathbb{H}_\gamma\to\mathscr{L}_2(\mathcal{V}_0,\mathbb{H}_\gamma)$ are the Nemytskii operators associated with $b$ and $g$, respectively.
Here $\mathcal{V}_0$ is the closed space spanned by $\{\widehat{{\mathfrak{u}}_j \mu}\}_{j=1}^\infty$, and it is the reproducing kernel Hilbert space of $\mathcal{W}$.
The convergence analysis of the numerical scheme \eqref{eq:EEMmult} was carried out under the following assumption.

\begin{assumption}
\label{Asp:gamma-S}
Let $\gamma,\vartheta$ satisfy Assumption \ref{Asp:gamma} and let $z^2\vartheta(z)$ be uniformly bounded with respect to $z\in[0,\infty)$. Moreover, Assumption \ref{Asp:gamma} and \eqref{eq:IkTk} hold with $\gamma$ and $\vartheta$ replaced by $\sqrt{\gamma}$ and $\sqrt{\vartheta}$, respectively.
\end{assumption}

The following example provides $\gamma$ satisfying
Assumptions \ref{Asp:gamma} and \ref{Asp:gamma-S}.
  \begin{example}\label{Ex:gamma-S1}
 Let $\gamma(z,k)=\vartheta(z)$ for every $(z,k)\in\Gamma$, where
$\vartheta:[0,\infty)\to(0,\infty)$ is twice differentiable. Moreover, there exist  $c_0>0$, $\lambda>2$ and $\lambda_1>0$ such that
\begin{equation*}
\vartheta(z)=c_0z^{-\lambda}\quad \text{or}\quad\vartheta(z)=c_0e^{-\lambda_1(\sqrt{z}-\sqrt{2z_0})}\quad\text{for}\quad z\ge z_0,
\end{equation*}
 where $z_0\ge \max_{1\le i\le m_1}H(\mathbf{x}_i)+1$ is sufficiently large.
\end{example}

The following proposition collects two results from \cite{CS24}:
the convergence of the exponential Euler approximation \eqref{eq:EEMmult} for
fixed $\epsilon$, and the discrete fast-advection limit for fixed
$N$.

\begin{proposition}[{\cite[Theorems 1 \& 2]{CS24}}]\label{theo:MS-2D}
Let $b$ and $g$ be continuously differentiable with bounded derivatives, and let Assumption \ref{Asp:gamma-S} hold. 

(1) If
 $\int_{\R^2}|\chi(x)|^2\sqrt{\vartheta(H(x))}\ud x<\infty,\nabla\chi\in \mathbb{H}_\gamma $ and  ${\int_{\R^2}|\xi|^2\ud\mu<\infty}$, 
then there exists $C>0$ such that for any $\epsilon\in(0,1]$ and $n=0,1,\ldots,N$,
 \begin{equation*}
\mathbb{E}\left[\|u_{\epsilon}(t_n)-U_{\epsilon}^{n}\|^2_{\mathbb{H}_\gamma}\right]
\le C \tau(1+\epsilon^{-2}).
\end{equation*}

(2) If \eqref{eq:Tneq0} 
holds, $\chi\in \mathbb{H}_\gamma$, and $\chi=(\chi^{\wedge})^{\vee}$, then
for any fixed $\tau>0$ and $p\ge2$,
\begin{equation*}
\lim_{\epsilon\to 0}\E\bigg[\sup_{1\le n\le N}\|U_{\epsilon}^{n}-(U^{n})^{\vee}\|_{\mathbb{H}_\gamma}^p\bigg]=0,
\end{equation*}
where $\{U^n\}_{n=0}^N$ is the exponential Euler approximation for \eqref{eq:intro}, given by  \eqref{eq:Un-intro}.
\end{proposition}

 Combining 
Propositions~\ref{tho:faa} and \ref{theo:MS-2D}, the remaining  ingredient in the commutative
diagram of Fig.~\ref{fig.doubleDiagram} is the strong convergence of
$
U^N$ to $u(T)$ in $L^2(\Omega,L^2_{\beta\gamma}(\Gamma))$,
which will be established in Section~\ref{S4}.
\section{Stochastic PDE on the noncompact {\color{black}metric graph}}\label{S3}
In this section, we recall the well-posedness framework for the stochastic PDE
\eqref{eq:intro} on the noncompact {\color{black}metric graph} $\Gamma$ and establish the regularity
estimates needed below. In particular, the optimal temporal H\"older regularity estimate of \eqref{eq:intro} will play a key role in the convergence analysis of the exponential Euler method \eqref{eq:Un-intro} in section~\ref{S4}.

According to \cite{FW12}, the differential operator $\mathcal{L}$ in \eqref{eq:intro} is defined by 
\begin{equation}\label{eq:barL}
	\mathcal{L} f(z,k)
	= \frac{1}{2\beta_k(z)} \frac{\ud}{\ud z}
	\Bigl( \alpha_k(z) \frac{\ud f}{\ud z}(z,k) \Bigr),
	\qquad \text{if }(z,k) \in \mathring{I}_k.
\end{equation} The coefficients in \eqref{eq:barL} are given by
\begin{equation}\label{eq:AT}
	\alpha_k(z)=\oint_{\textup{C}_k(z)} |\nabla  H (x)| \ud l_{z, k},\quad \beta_k(z)=\oint_{\textup{C}_k(z)} \frac{1}{|\nabla  H (x)|}\ud l_{z,k},
\end{equation}
where $\ud l_{z,k}$ is the length element on $\textup{C}_k(z)$.
The operator $\mathcal{L}$ is subject to the following gluing conditions: $f$ is required to be continuous at each interior vertex $O_i $ and to satisfy the  Kirchhoff condition:
\begin{equation}\label{eq:glue}
	\sum_{k: I_k \sim O_i} 
	\alpha_k(H (\mathbf{x}_i))\ud_k f(H (\mathbf{x}_i),k) = 0 
	\quad \text{for each interior vertex } O_i,
\end{equation}
where the oriented derivative $\ud_kf(H (\mathbf{x}_i),k):=\Phi_{k,i}\frac{\ud}{\ud z}f(H (\mathbf{x}_i),k)$ with $\Phi_{k,i}=+1$ if $O_i$  is incident to $I_k$ and the $z$-coordinate decreases along $I_k$ toward $O_i$, and $\Phi_{k,i}=-1$ otherwise.

Two functions $\varphi$ and $\psi$ are said to be asymptotically equivalent as $x\to x_0$ if $\lim_{x\to x_0}\varphi(x)/\psi(x)=1$, in which case we write $\varphi(x)\sim\psi(x)$ as $x\to x_0$. 
\begin{remark}\label{rem:AT}
	By \cite[Lemma 1.1]{FW12}, the functions $\beta_k$ and $\alpha_k$ are continuously differentiable in the interior $\mathring{I}_k$ of $I_k$. 
	In contrast, $\alpha_k$ and $\beta_k$ may be degenerate or singular near the vertices. More precisely, if $(z,k)$ approaches an endpoint of an edge $I_k$, corresponding to a vertex $O_i=(H (\mathbf{x}_i),k)$, then as $(z,k)\to O_i$ (see \cite[pp.~266-267]{FW12} and \cite[p.~501]{CF19}),
	\begin{align*}%
			&\alpha_k(z)\sim\begin{cases} \textup{const}\cdot|z-H (\mathbf{x}_i)|,\quad &\text{if $\mathbf{x}_i$ is a local extremum point},\\
					\textup{const},\quad& \text{if $\mathbf{x}_i$ is a saddle point},\\
					\textup{const}\cdot z,\quad &\text{if $O_i=O_\infty$};
				\end{cases}\\%
			&\beta_k(z)\sim\begin{cases}\textup{const},&\text{if $\mathbf{x}_i$ is a local extremum point},\\
					\textup{const}\cdot|\log|z- H (\mathbf{x}_i)||,& \text{if $\mathbf{x}_i$ is a saddle point},\\
					\textup{const},& \text{if $O_i=O_\infty$}.
				\end{cases}
		\end{align*}
	Here \textup{const} denotes some positive constant depending on $k$ and $O_i$, and may differ from line to line.
\end{remark}

 \subsection{Well-posedness}\label{S3.1}
In this subsection, we recall from \cite{CKS26} the well-posedness framework
for \eqref{eq:intro}.
Denote by  $W^{1,2}_{\alpha\gamma,\beta\gamma}(\Gamma)$ the subspace of $L^2_{\beta\gamma}(\Gamma)$ endowed with the inner product 
\begin{equation*}
\langle \phi,\psi\rangle_{W^{1,2}_{\alpha\gamma,\beta\gamma}(\Gamma)}:=\sum_{k=1}^m\int_{I_k} \Big(\phi(z,k)\overline{\psi(z,k)} \beta_k(z) + \alpha_k(z) \frac{\mathrm{d}}{\mathrm{d} z} \phi(z,k)\frac{\mathrm{d}}{\mathrm{d} z} \overline{\psi(z,k)}\Big) \gamma_k(z) \mathrm{d}z.
\end{equation*}
	Let $W^{1,2}_{\mathrm{cont}}(\Gamma)$ denote the subspace of $W^{1,2}_{\alpha\gamma,\beta\gamma}(\Gamma)$ consisting of functions that are continuous at each interior vertex. By Remark \ref{rem:AT}, $W^{1,2}_{\mathrm{cont}}(\Gamma)$ is a closed subspace of $W^{1,2}_{\alpha\gamma,\beta\gamma}(\Gamma)$ (see \cite[Section 3]{CKS26} for more details). Recall that, under Assumption \ref{Asp:gamma}, the SPDE \eqref{eq:intro} on $\Gamma$ is well posed and serves as the limiting equation for the multiscale model \eqref{eq:SRDA-intro}; see Proposition \ref{tho:faa}. In fact, for the well-posedness of \eqref{eq:intro}, it suffices to impose the following assumption.

\begin{assumption}\label{asp:aT}
	For every edge $I_k$, the function $\gamma_k$ is locally absolutely
	continuous on $\mathring I_k$. Moreover, there exists a constant
	$\kappa>0$ such that
	\begin{equation*}%
		\alpha_k(z)|\gamma_k^\prime(z)|^2\le \kappa \beta_k(z)|\gamma_k(z)|^2\quad \text{for }(z,k)\in \mathring{I}_k,~ k=1,\ldots,m.
	\end{equation*}
\end{assumption}

By Remark~\ref{rem:AT}, Assumption~\ref{asp:aT} is weaker than Assumption~\ref{Asp:gamma}. Moreover, Assumption~\ref{Asp:gamma-S} implies both Assumption~\ref{Asp:gamma} and Assumption~\ref{asp:aT}.

\begin{lemma}{\cite[Theorem 3.3]{CKS26}}\label{thm:3.3}
	Under Assumption \ref{asp:aT}, for any $\lambda>\frac{1}{8}\kappa$, the operator $-\lambda I+\mathcal{L}$ generates an analytic contraction semigroup on $L^2_{\beta\gamma}(\Gamma)$, where
	\begin{align}\label{eq:DL}
D(\mathcal L)=\Bigl\{f\in W^{1,2}_{\mathrm{cont}}(\Gamma):
 f\text{ satisfies }\eqref{eq:glue},
 \alpha_k \frac{\ud f}{\ud z} \text{ is locally absolutely}\\ \text{ continuous on each } I_k
\text{ and } \mathcal Lf
 \in L^2_{\beta\gamma}(\Gamma)\Bigr\}.\notag
	\end{align}
\end{lemma}
\begin{proof}
We sketch the main ingredients of the proof, which will also be used later in the numerical error analysis, and refer to \cite{CKS26} for full details.
For $\lambda\in\mathbb{R}$, consider the sesquilinear form 
	$a_{\lambda}: W^{1,2}_{\mathrm{cont}}(\Gamma)\times W^{1,2}_{\mathrm{cont}}(\Gamma)\to \mathbb{C}$ defined by
		\begin{align*}%
		a_{\lambda}(\phi,\psi)
		&=\lambda \sum_{k=1}^m \int_{I_k}
		\phi(z,k)\overline{\psi(z,k)}\beta_k(z)\gamma_k(z)\,\mathrm{d}z\\
		&\quad+\frac{1}{2}\sum_{k=1}^m\int_{I_k}
		\alpha_k(z)\frac{\mathrm{d}}{\mathrm{d}z}\phi(z,k)
		\frac{\mathrm{d}}{\mathrm{d}z}
		\bigl(\overline{\psi(z,k)}\gamma_k(z)\bigr)\,\mathrm{d}z.
		\end{align*}	
	To see the origin of the threshold, set
	\begin{equation*}
	X:=\bigg(\sum_{k=1}^m\int_{I_k}\alpha_k|\phi'|^2\gamma_k\,\ud z\bigg)^{1/2},
	\qquad
	Y:=\bigg(\sum_{k=1}^m\int_{I_k}|\phi|^2\beta_k\gamma_k\,\ud z\bigg)^{1/2}.
	\end{equation*}
Assumption~\ref{asp:aT} and the Cauchy--Schwarz inequality give
	\begin{equation*}
	\bigg|\frac12\sum_{k=1}^m\int_{I_k}
	\alpha_k\phi'\overline\phi\,\gamma_k'\,\ud z\bigg|
	\leq \frac{\sqrt\kappa}{2}XY.
	\end{equation*}
Consequently,
	$\operatorname{Re}a_\lambda(\phi,\phi)
	\geq \frac12X^2-\frac{\sqrt\kappa}{2}XY+\lambda Y^2$,
	and this quadratic form is coercive precisely when $\lambda>\kappa/8$.
	The same estimate, without taking real parts, proves continuity.
	For $\lambda > \frac{1}{8} \kappa$, where $\kappa$ is given in Assumption \ref{asp:aT},	the sesquilinear form $a_\lambda$ is densely defined, continuous, closed, and accretive. Moreover, 
	the operator associated with $a_\lambda$ via the standard form-operator correspondence is precisely $\lambda I - \mathcal{L}$ with domain given by \eqref{eq:DL}, namely 
	\begin{equation*}
	a_{\lambda}(\phi,\psi)=\langle (\lambda I - \mathcal{L})\phi,\psi\rangle_{L^2_{\beta\gamma}(\Gamma)},\quad \phi\in D(\mathcal{L}),\psi\in W^{1,2}_{\mathrm{cont}}(\Gamma).
	\end{equation*}
	In particular, for any $\lambda>\frac{1}{8}\kappa$, there exist positive constants $C(\lambda)$ and $c(\lambda)$ such that 
	\begin{align}\label{eq:lower} \Re a_\lambda(\phi,\phi) &\ge c(\lambda)\|\phi\|_{W^{1,2}_{\alpha\gamma,\beta\gamma}(\Gamma)}^2\quad\forall~\phi\in W^{1,2}_{\mathrm{cont}}(\Gamma),\\\label{eq:aconti} |a_{\lambda}(\phi,\psi)|&\le C(\lambda)\|\phi\|_{W^{1,2}_{\alpha\gamma,\beta\gamma}(\Gamma)}\|\psi\|_{W^{1,2}_{\alpha\gamma,\beta\gamma}(\Gamma)}\quad\forall~ \phi,\psi\in W^{1,2}_{\mathrm{cont}}(\Gamma). 
	\end{align}
Since the sesquilinear form $a_\lambda$ associated with $\lambda I-\mathcal{L}$ is densely defined, continuous, closed, and accretive, it follows from \cite[Theorems 1.51 \& 1.52]{OE05} that for any $\lambda > \frac{1}{8}\kappa$, the operator $-\lambda I+\mathcal{L}$ generates an analytic contraction semigroup on $L^2_{\beta\gamma}(\Gamma)$.
\end{proof}

By Lemma \ref{thm:3.3}, $\mathcal{L}$ generates an analytic semigroup $\{S(t)\!\!:=\!e^{t\mathcal{L}}\}_{t\ge0}$ on $L^2_{\beta\gamma}(\Gamma)$ and
	\begin{equation}\label{eq:Stbound}
		\|S(t)\|_{\mathscr{L}(L^2_{\beta\gamma}(\Gamma))}\le e^{\frac{1}{8}\kappa t},\quad t\ge0.
	\end{equation}
Let $\mathcal{U}_0$ be the reproducing kernel Hilbert space of $W$, and let $K_\Gamma$ denote the covariance kernel of $W(1)$. The projected functions $\{(\widehat{{\mathfrak{u}}_j\mu})^\wedge\}_{j\geq1}$ generate this covariance, but projection need not preserve orthonormality. If $\{e_\ell\}_{\ell\geq1}$ is any orthonormal basis of $\mathcal U_0$, then the reproducing-kernel identity and \eqref{eq:ujmu2} give
\begin{equation}\label{eq:KGamma-diagonal}
 \sum_{\ell=1}^\infty |e_\ell(z,k)|^2
 =K_\Gamma((z,k),(z,k))
 =\sum_{j=1}^\infty
 |(\widehat{{\mathfrak u}_j\mu})^\wedge(z,k)|^2
 \leq \mu(\mathbb R^2).
\end{equation}
Let
$B$ and $G$ be the Nemytskii operators
associated with $b$ and $g$, respectively, namely, 
\begin{equation}\label{eq:BG}
	B(v_1)(z,k)=b(v_1(z,k)),\quad G(v_1)(v_2)(z,k)=g(v_1(z,k))v_2(z,k),\quad (z,k)\in\Gamma,
\end{equation}
where $v_1\in L^2_{\beta\gamma}(\Gamma)$ and $v_2\in \mathcal{U}_0$. In fact,
the mappings $B:L^2_{\beta\gamma}(\Gamma)\to L^2_{\beta\gamma}(\Gamma)$
and $G:L^2_{\beta\gamma}(\Gamma)\to \mathscr{L}_2(\mathcal{U}_0,L^2_{\beta\gamma}(\Gamma))$ are Lipschitz continuous and of linear growth. 
More precisely, there exists $C>0$ such that for any $v_1,v_2\in L^2_{\beta\gamma}(\Gamma)$,
\begin{gather}\label{eq:Lin}
	\|B(v_1)\|_{L^2_{\beta\gamma}(\Gamma)}+\|G(v_1)\|_{\mathscr{L}_2(\mathcal{U}_0,L^2_{\beta\gamma}(\Gamma))}\le C(1+\|v_1\|_{L^2_{\beta\gamma}(\Gamma)}),\\\label{eq:Lip}
	\|B(v_1)-B(v_2)\|_{L^2_{\beta\gamma}(\Gamma)}+\|G(v_1)-G(v_2)\|_{\mathscr{L}_2(\mathcal{U}_0,L^2_{\beta\gamma}(\Gamma))}\le C\|v_1-v_2\|_{L^2_{\beta\gamma}(\Gamma)}.
\end{gather}
We use $G$ to illustrate \eqref{eq:Lin} and \eqref{eq:Lip}.
Since $\mu(\R^2)<\infty$, by \eqref{eq:KGamma-diagonal} and the Lipschitz continuity of $g$, for any $v_1,v_2\in L^2_{\beta\gamma}(\Gamma)$,
\begin{align*} 
&\|G(v_1)-G(v_2)\|^2_{\mathscr{L}_2(\mathcal{U}_0,L^2_{\beta\gamma}(\Gamma))}\\
&=\sum_{k=1}^m\int_{I_k}
 |g(v_1(z,k))-g(v_2(z,k))|^2
 K_\Gamma((z,k),(z,k))\gamma_k(z)\beta_k(z)\ud z\\
&\le\mu(\R^2)\sum_{k=1}^m\int_{I_k}|g(v_1(z,k))-g(v_2(z,k))|^2\gamma_k(z)\beta_k(z)\ud z
\le C\|v_1-v_2\|_{L^2_{\beta\gamma}(\Gamma)}^2.
\end{align*}
In a similar manner, we obtain from \eqref{eq:IkTk} and the linear growth of $g$ that 
\begin{equation*}
\|G(v_1)\|_{\mathscr{L}_2(\mathcal{U}_0,L^2_{\beta\gamma}(\Gamma))}^2\le \mu(\R^2)\sum_{k=1}^m\int_{I_k}|g(v_1(z,k))|^2\gamma_k(z)\beta_k(z)\ud z\le C(1+\|v_1\|_{L^2_{\beta\gamma}(\Gamma)}^2).
\end{equation*} 
In view of \eqref{eq:Lip} and \eqref{eq:Lin}, by standard theory for semilinear stochastic evolution equations (see, e.g., \cite[Chapter 7]{DZ14} and \cite{CHS26}), we have the following result.

\begin{proposition}
Under Assumption \ref{asp:aT},
for every initial value $\chi^\wedge\in L^2_{\beta\gamma}(\Gamma)$,
there exists a unique mild solution
$u=\{u(t)\}_{t\in[0,T]}$ to \eqref{eq:intro} in $L^p(\Omega,\mathcal{C}([0,T];L^2_{\beta\gamma}(\Gamma)))$ such that for any $1\le p<\infty$,
\begin{equation}\label{eq:L2}
	\E\bigg[\sup_{t\in[0,T]}\|u(t)\|^p_{L^2_{\beta\gamma}(\Gamma)}\bigg]\le C(p,T)\Big(1+\|\chi^\wedge\|^p_{L^2_{\beta\gamma}(\Gamma)}\Big).
\end{equation}
\end{proposition}
\subsection{Regularity estimates} 
Fix $\lambda>\frac{1}{8}\kappa$ and set
\(
\mathcal{A}_\lambda:=\lambda I-\mathcal{L}.
\)
By Lemma \ref{thm:3.3}, $-\mathcal{A}_\lambda$ generates an analytic semigroup
$\{
S_\lambda(t):=e^{-t\mathcal{A}_\lambda},
 t\geq0\}$
on $L^2_{\beta\gamma}(\Gamma)$. Moreover, \cite[Proposition 5.33]{EN00} yields the continuous embedding
\(
D(\mathcal{A}_\lambda^{\varrho_2})
\hookrightarrow
D(\mathcal{A}_\lambda^{\varrho_1}),
0<\varrho_1<\varrho_2<1.  
\) %
The following lemma shows that $S_\lambda$ enjoys the standard parabolic smoothing properties in the weighted space $L^2_{\beta\gamma}(\Gamma)$, despite the noncompactness of $\Gamma$ and the degeneracy of $\mathcal{L}$ near the vertices.%

\begin{lemma}\label{lem:smo}
	Let Assumption \ref{asp:aT} hold and $\lambda>\frac{1}{8}\kappa$. Then for any $\varrho>0$,
	\begin{equation}\label{eq:smo1}
		\|{\mathcal{A}_\lambda^{\varrho}} S_\lambda(t)\|_{\mathscr{L}(L^2_{\beta\gamma}(\Gamma))}\le C(\lambda,\varrho ) t^{-\varrho },\quad t>0.
		\end{equation}
		Moreover, for any $\varrho \in (0,1)$,
		\begin{equation}
		 \label{eq:smo2}
		\|{\mathcal{A}_\lambda^{-\varrho}}(S_\lambda(t)-I)\|_{\mathscr{L}(L^2_{\beta\gamma}(\Gamma))}\le C(\lambda,\varrho )t^{\varrho },\quad t>0.
	\end{equation}
\end{lemma}
\begin{proof}
	For a fixed  $\lambda > \tfrac{1}{8}\kappa$,
	combining \eqref{eq:lower}  and \eqref{eq:aconti} yields   
	\begin{equation*}
		|\langle \mathcal{A}_\lambda \phi,\phi\rangle_{L^2_{\beta\gamma}(\Gamma)}|\le c(\lambda)^{-1}C(\lambda)\Re \langle \mathcal{A}_\lambda \phi,\phi\rangle_{L^2_{\beta\gamma}(\Gamma)},\quad \phi\in D(\mathcal{A}_\lambda)=D(\mathcal{L}),
	\end{equation*}
	where $\Re$ denotes the real part.
Hence, the numerical range $\mathcal{N}(\mathcal{A}_\lambda )$ of $\mathcal{A}_\lambda $ (see \cite[Definition 1.26]{OE05}), defined by
	\begin{equation*}
	\mathcal{N}(\mathcal{A}_\lambda ):=\Big\{\langle \mathcal{A}_\lambda \phi,\phi\rangle_{L^2_{\beta\gamma}(\Gamma)}:\phi\in D(\mathcal{A}_\lambda ), \|\phi\|_{L^2_{\beta\gamma}(\Gamma)}= 1 \Big\},
	\end{equation*}
	is contained in the sector $\Sigma_0(\theta_0):=\{y\in\mathbb{C},y\neq 0,0\le | \arg y|\le \theta_0\}$ where $\theta_0=\arccos(c(\lambda)/C(\lambda))\in[0,{\pi}/{2}).$
	Let $\theta_1\in (\theta_0,{\pi}/{2})$ and denote $\Sigma_\pi(\theta_1):=\{y\in\mathbb{C},\theta_1\le | \arg y|\le \pi\}$. Then by \cite[Proposition C.3.1]{HM06}, the resolvent set $\rho(\mathcal{A}_\lambda )$ of $\mathcal{A}_\lambda $ contains $\Sigma_\pi(\theta_1)$, and for any $y\in \Sigma_\pi(\theta_1)$,
	\begin{equation}\label{eq:analytic}
		\|(y I-\mathcal{A}_\lambda )^{-1}\|_{\mathscr{L}(L^2_{\beta\gamma}(\Gamma))}\!\le\! \frac{1}{\textup{dist}(y,\overline{\mathcal{N}(\mathcal{A}_\lambda )})}\!\le\!\frac{1}{\textup{dist}(y,\Sigma_0(\theta_0))}\le \frac{1}{\sin(\theta_1-\theta_0)}\frac{1}{|y|},
	\end{equation}
	where $\textup{dist}(y,A)$ denotes the distance between $y\in\mathbb{C}$ and a set $A\subset \mathbb{C}$.
	\iffalse
	\begin{figure}[htbp]
\centering
\begin{tikzpicture}[scale=2, >=stealth, line cap=round, line join=round]

\def\thetaZero{25}
\def\thetaOne{50}
\def\argY{65}

\fill[blue!8] 
    (0,0) -- (\thetaZero:1.8) 
    arc (\thetaZero:-\thetaZero:1.8) -- cycle;

\draw[blue!60, thick] 
    (0,0) -- (\thetaZero:1.8);
\draw[blue!60, thick] 
    (0,0) -- (-\thetaZero:1.8);

\node[blue!70!black] at (1.55,0.25) {$\Sigma_0(\theta_0)$};

\draw[->, gray!70] (-0.5,0) -- (2,0) node[right] {$\mathrm{Re}$};
\draw[->, gray!70] (0,-1) -- (0,1.2) node[above] {$\mathrm{Im}$};

\draw[dashed, gray] (0,0) -- (\thetaOne:1.7);
\node[gray] at (\thetaOne:1.85) {$\theta_1$};

\coordinate (O) at (0,0);
\coordinate (Y) at (\argY:1.5);

\fill (Y) circle (0.6pt) node[above right] {$y$};
\draw[thick] (O) -- (Y) node[midway, left] {$|y|$};

\path (O) -- (\thetaZero:2) coordinate (Ray);
\coordinate (Foot) at ($(O)!(Y)!(Ray)$);

\draw[red!70!black, thick] (Y) -- (Foot) node[midway, right] {$\textup{dist}(y,\Sigma_0(\theta_0))$};

\pic[draw, angle radius=5pt] {right angle = Y--Foot--O};

\fill (Foot) circle (0.5pt);

\draw[<->, blue!60] (\thetaZero:0.6) arc (\thetaZero:\argY:0.6);
\node[blue!60!black] at (35:0.8) {\small $\ge \theta_1 - \theta_0$};

\node at (1.85,-0.1) {};

\end{tikzpicture}
\caption{Geometric illustration of distance between $y\in\Sigma_\pi(\theta_1)$ and  the sector $\Sigma_0(\theta_0)$.}\label{Fig:ill}
\end{figure}
\fi
Consider the contour $\mathcal{Y}=\{y:y=r e^{\pm \mathbf{i}\theta_1},r\ge0\}$. From the Dunford--Taylor representation of fractional powers of sectorial operators, as well as \eqref{eq:analytic}, for any $\varrho >0$ we have
	\begin{align*}
		\|{\mathcal{A}_\lambda^{\varrho}} S_\lambda(t)\phi\|_{L^2_{\beta\gamma}(\Gamma)}&=\left\|\frac{1}{2\pi \mathbf{i}}\int_{\mathcal{Y}} y^\varrho  e^{-yt}(yI-\mathcal{A}_\lambda )^{-1}\phi\ud  y\right\|_{L^2_{\beta\gamma}(\Gamma)}\\\notag
		&\le C\|\phi\|_{L^2_{\beta\gamma}(\Gamma)}\int_0^\infty r^{\varrho -1} e^{-r t\cos\theta_1}\ud r\le C t^{-\varrho }\|\phi\|_{L^2_{\beta\gamma}(\Gamma)},
	\end{align*}
	where $C:=C(\varrho,\lambda)>0$ is independent of $\phi\in L^2_{\beta\gamma}(\Gamma)$. 
	This finishes the proof of \eqref{eq:smo1}. Finally, \eqref{eq:smo2} follows from \eqref{eq:smo1} and the identity 
	$(I-S_\lambda(t))\phi=\int_0^t\mathcal{A}_\lambda S_\lambda(s)\phi\ud s$.
\end{proof}

By introducing the shifted process $v_\lambda(t):=e^{-\lambda t}u(t)$, we have
\begin{equation}\label{eq:vlambda}
	\ud v_\lambda(t)=-\mathcal{A}_\lambda v_\lambda(t)\ud t
	+e^{-\lambda t}B(u(t))\ud t
	+e^{-\lambda t}G(u(t))\ud W(t),\quad t\in(0,T]
\end{equation}
with the initial value $v_\lambda(0)=u(0)=\chi^\wedge$. 

\begin{lemma}\label{lem:uHol1}
Suppose that Assumption \ref{asp:aT} holds and that
$\chi^\wedge\in L^2_{\beta\gamma}(\Gamma)$. Then, for every
$\varrho\in(0,\frac{1}{2})$ and $p\geq2$, there exists $C>0$ such that
\begin{equation}\label{eq:uHol-sing}
\|u(t)-u(s)\|_{L^p(\Omega,L^2_{\beta\gamma}(\Gamma))}
\leq
C(1+s^{-\varrho})(t-s)^\varrho,
\qquad 0<s\leq t\leq T.
\end{equation}
\end{lemma}

\begin{proof}		
	 For any $0\le s\le t\le T$,
	\begin{gather}\label{eq:v-v}
		v_\lambda(t)-v_\lambda(s)
		=(S_\lambda(t-s)-I)v_\lambda(s)+\mathcal{K}(t,s),\\\notag
		\mathcal{K}(t,s):=\int_s^tS_\lambda(t-r)e^{-\lambda r}B(u(r))\ud r
		+\int_s^tS_\lambda(t-r)e^{-\lambda r}G(u(r))\ud W(r).
	\end{gather}
	Then from \eqref{eq:Stbound}, the linear growth of $B$ and $G$ in \eqref{eq:Lin}, and \eqref{eq:L2}, we deduce
	\begin{align}\label{eq:Kt}
		\E\left[\|\mathcal{K}(t,s)\|_{L^2_{\beta\gamma}(\Gamma)}^p\right]&\le C\E\left[\Big|\int_s^t\|S_\lambda(t-r)e^{-\lambda r}B(u(r))\|_{L^2_{\beta\gamma}(\Gamma)}\ud r\Big|^p\right]\\\notag
		&\quad+C\E\left[\Big|\int_s^t\|S_\lambda(t-r)e^{-\lambda r}G(u(r))\|_{\mathscr{L}_2(\mathcal{U}_0,L^2_{\beta\gamma}(\Gamma))}^2\ud r\Big|^{\frac{p}{2}}\right]\\\notag
		&\le C(t-s)^{\frac{p}{2}}.
	\end{align}
	The mild formulation of \eqref{eq:vlambda} reads
	$v_\lambda(t)=S_\lambda(t)\chi^\wedge
		+\mathcal{K}(t,0)$
	for any $t\in[0,T]$. Let $p\ge2$ and $\varrho\in(0,\frac{1}{2})$ throughout this proof. Then it follows from the definition of $\mathcal{K}(t,0)$ and the Burkholder inequality that
		\begin{align}\label{eq:Atheta}
		&\E\left[\|{\mathcal{A}_\lambda^{\varrho}} v_\lambda(t)\|_{L^2_{\beta\gamma}(\Gamma)}^p\right]
		\le C\|{\mathcal{A}_\lambda^{\varrho}} S_\lambda(t)\chi^\wedge\|_{L^2_{\beta\gamma}(\Gamma)}^p+\mathcal{J}(t),\\
		&\mathcal{J}(t):=C\E\bigg[\Big|\int_0^t\|{\mathcal{A}_\lambda^{\varrho}}S_\lambda(t-s) e^{-\lambda s}B(u(s))\|_{L^2_{\beta\gamma}(\Gamma)}\ud s\Big|^p\bigg]\\\notag
		&\qquad\qquad+C\E\left[\Big|\int_0^t\|{\mathcal{A}_\lambda^{\varrho}}S_\lambda(t-s)e^{-\lambda s}G(u(s))\|_{\mathscr{L}_2(\mathcal{U}_0,L^2_{\beta\gamma}(\Gamma))}^2\ud s\Big|^{\frac{p}{2}}\right].
	\end{align}
	Furthermore, according to Lemma \ref{lem:smo} and the linear growth of $B$ and $G$ in \eqref{eq:Lin}, we obtain that for any $t\in (0,T]$,
	\begin{align*}
		\mathcal{J}(t)&\le C\E\bigg[\Big|\int_0^t(t-s)^{-\varrho}(1+\|u(s)\|_{L^2_{\beta\gamma}(\Gamma)})\ud s\Big|^p\bigg]\\\notag
		&\quad+C\E\left[\Big|\int_0^t(t-s)^{-2\varrho}(1+\|u(s)\|_{L^2_{\beta\gamma}(\Gamma)}^2)\ud s\Big|^{\frac{p}{2}}\right]
		\le C(p,\varrho,\chi,T),
	\end{align*}
	where the last step is due to 
 \eqref{eq:L2}, $\varrho\in(0,\frac{1}{2})$, and the Minkowski inequality.
Hence,
		\begin{equation}\label{eq:Alarh}
		\E\left[\|{\mathcal{A}_\lambda^{\varrho}} v_\lambda(t)\|_{L^2_{\beta\gamma}(\Gamma)}^p\right]
		\le C\|{\mathcal{A}_\lambda^{\varrho}} S_\lambda(t)\chi^\wedge\|_{L^2_{\beta\gamma}(\Gamma)}^p+C.
		\end{equation}
Furthermore, using Lemma \ref{lem:smo} and the assumption $\chi^\wedge\in L^2_{\beta\gamma}(\Gamma)$ immediately yields			\begin{equation}\label{eq:Ala}
		\|{\mathcal{A}_\lambda^{\varrho}} v_\lambda(t)\|_{L^p(\Omega,L^2_{\beta\gamma}(\Gamma))}
		\le C(1+t^{-\varrho}),\quad 0<t\le T.
		\end{equation}
		Recall that by \eqref{eq:v-v} and \eqref{eq:Kt}, for any $0\le s\le t\le T$,
			\begin{equation}\label{eq:vtvs}
		\|v_\lambda(t)-v_\lambda(s)\|_{L^p(\Omega,L^2_{\beta\gamma}(\Gamma))}\le \|(S_\lambda(t-s)-I)v_\lambda(s)\|_{L^p(\Omega,L^2_{\beta\gamma}(\Gamma))}+C(t-s)^{\frac{1}{2}},
	\end{equation}
	which together with Lemma \ref{lem:smo}, \eqref{eq:Ala}, and $u(t)=e^{\lambda t}v_\lambda(t)$ completes the proof.	
\end{proof}

In the following lemma, we present the fractional-domain regularity of the mild solution to \eqref{eq:intro}. The
domain of the form $a_\lambda$ associated with $\mathcal{A}_\lambda$ is $W^{1,2}_{\mathrm{cont}}(\Gamma)$, but the equality
$D(\mathcal A_\lambda^{1/2})=W^{1,2}_{\mathrm{cont}}(\Gamma)$ does not follow from the present nonsymmetric form
setting. We therefore state the needed condition directly as
$\chi^\wedge\in D(\mathcal A_\lambda^{1/2})$. The readily checked stronger
condition $\chi^\wedge\in D(\mathcal L)$ implies it because
$D(\mathcal L)=D(\mathcal A_\lambda)\subset
D(\mathcal A_\lambda^{1/2})$, with $D(\mathcal L)$ given in \eqref{eq:DL}.
\begin{lemma}\label{lem:Holder}
	Let Assumption \ref{asp:aT} hold and the initial value $\chi^\wedge\in D({\mathcal{A}_\lambda^{1/2}})$. Then for any  $p\ge2$, there exists $C>0$ such that for all $0\le t\le T$,
	\begin{equation}\label{eq:Holder}
		\E\left[\|{\mathcal{A}_\lambda^{1/2}}u(t)\|_{L^2_{\beta\gamma}(\Gamma)}^p\right]\le C.
	\end{equation}
\end{lemma}
\begin{proof}
Let $p\ge2$.
	First, invoking \eqref{eq:Stbound}, \eqref{eq:Alarh} and $\chi^\wedge\in D({\mathcal{A}_\lambda^{1/2}})\subset D({\mathcal{A}_\lambda^{\varrho}})$ for $\varrho\in(0,\frac{1}{2})$, we have
			\begin{equation*}
		\E\left[\|{\mathcal{A}_\lambda^{\varrho}} v_\lambda(t)\|_{L^2_{\beta\gamma}(\Gamma)}^p\right]
		\le C\|{\mathcal{A}_\lambda^{\varrho}}\chi^\wedge\|_{L^2_{\beta\gamma}(\Gamma)}^p+C\le C.
		\end{equation*}
Together with \eqref{eq:vtvs} and Lemma \ref{lem:smo}, this implies that for any $\varrho\in(0,\frac{1}{2})$,
			\begin{equation}\label{eq:vHol}
		\|v_\lambda(t)-v_\lambda(s)\|_{L^p(\Omega,L^2_{\beta\gamma}(\Gamma))}\le C_{\varrho}(t-s)^{\varrho}+C(t-s)^{\frac{1}{2}}\le C_{\varrho}(t-s)^{\varrho}.
	\end{equation}	
	Combining \eqref{eq:vHol} and $u(t)=e^{\lambda t}v_\lambda(t)$, we obtain that for any $\varrho\in(0,\frac{1}{2})$ and $p\ge2$,
	\begin{equation}\label{eq:uHol}
		\|	u(t)-u(s)\|_{L^p(\Omega,L^2_{\beta\gamma}(\Gamma))}\le C(t-s)^{\varrho}.
	\end{equation}
By Lemma~\ref{thm:3.3}, $-\mathcal A_\lambda$ generates an analytic, and hence strongly continuous, contraction semigroup
$
\{S_\lambda(t)=e^{-t\mathcal A_\lambda}\}
$
on the Hilbert space $L^2_{\beta\gamma}(\Gamma)$. Thus the hypotheses of \cite[Proposition~A.19 and Remark~A.20]{DZ14} are satisfied. Proposition~A.21 of \cite{DZ14} then yields
\begin{equation}\label{eq:smooth}
\int_0^T
\bigl\|\mathcal A_\lambda^{1/2}S_\lambda(s)\phi\bigr\|_{L^2_{\beta\gamma}(\Gamma)}^2\,\ud s
\le C\|\phi\|_{L^2_{\beta\gamma}(\Gamma)}^2
\qquad
\forall~\phi\in L^2_{\beta\gamma}(\Gamma).
\end{equation}
This estimate is a consequence of the contraction-semigroup property and does not require $\mathcal A_\lambda$ to be self-adjoint.

	Now we are in a position to prove \eqref{eq:Holder}, whose proof is inspired by \cite{KR14}.
Similar to \eqref{eq:Atheta}, applying the Burkholder inequality to the stochastic convolution and only then using the decomposition
$
G(u(s))=G(u(t))+(G(u(s))-G(u(t)))
$
inside the Hilbert--Schmidt norm, we obtain
	\begin{align}\label{eq:theta=1/2}%
		&\E\left[\|{\mathcal{A}_\lambda^{1/2}} v_\lambda(t)\|_{L^2_{\beta\gamma}(\Gamma)}^p\right]\le C+C\mathcal{I}_1+C\mathcal{I}_2+C\mathcal{I}_3,\\\notag
		\mathcal{I}_1&:=\E\bigg[\Big|\int_0^t\|{\mathcal{A}_\lambda^{1/2}}S_\lambda(t-s) e^{-\lambda s}B(u(s))\|_{L^2_{\beta\gamma}(\Gamma)}\ud s\Big|^p\bigg]\\\notag
		\mathcal{I}_2&:=\E\left[\Big|\int_0^t\|{\mathcal{A}_\lambda^{1/2}}S_\lambda(t-s)e^{-\lambda s}G(u(t))\|_{\mathscr{L}_2(\mathcal{U}_0,L^2_{\beta\gamma}(\Gamma))}^2\ud s\Big|^{\frac{p}{2}}\right]\\\notag
		\mathcal{I}_3&:=\E\left[\Big|\int_0^t\|{\mathcal{A}_\lambda^{1/2}}S_\lambda(t-s)e^{-\lambda s}(G(u(t))-G(u(s)))\|_{\mathscr{L}_2(\mathcal{U}_0,L^2_{\beta\gamma}(\Gamma))}^2\ud s\Big|^{\frac{p}{2}}\right].
		\end{align}
	By Lemma \ref{lem:smo}, the linear growth of $B$, and \eqref{eq:L2}, we arrive at
	\begin{equation*}
	\begin{aligned}
		\mathcal{I}_1&\le C\E\bigg[\Big|\int_0^t(t-s)^{-\frac{1}{2}}(1+\|u(s)\|_{L^2_{\beta\gamma}(\Gamma)})\ud s\Big|^p\bigg]\\
		&\le C\bigg|\int_0^t(t-s)^{-\frac{1}{2}}(1+\|u(s)\|_{L^p(\Omega,L^2_{\beta\gamma}(\Gamma))})\ud s\bigg|^p\le C.
	\end{aligned}
	\end{equation*}
Applying \eqref{eq:smooth} to $G(u(t))e_\ell$ for an orthonormal basis $\{e_\ell\}$ of $\mathcal U_0$, and then summing by Tonelli's theorem, the linear growth of $G$, and \eqref{eq:L2},
	\begin{equation*}
		\mathcal{I}_2
		\le C\E\left[\|G(u(t))\|_{\mathscr{L}_2(\mathcal{U}_0,L^2_{\beta\gamma}(\Gamma))}^p\right]\le C\left(1+\|u(t)\|_{L^p(\Omega,L^2_{\beta\gamma}(\Gamma))}^p\right)\le C.
	\end{equation*}
	By Lemma \ref{lem:smo}, the Lipschitz continuity of $G$ in \eqref{eq:Lip}, the
Minkowski inequality, and \eqref{eq:uHol} with $\varrho=\frac{1}{4}$,
	\begin{align*}
		\mathcal{I}_3&\le	C\E\left[\Big|\int_0^t(t-s)^{-1}\|G(u(s))-G(u(t))\|_{\mathscr{L}_2(\mathcal{U}_0,L^2_{\beta\gamma}(\Gamma))}^2\ud s\Big|^{\frac{p}{2}}\right]\\
		&\le C\Big|\int_0^t(t-s)^{-\frac{1}{2}}\ud s\Big|^{\frac{p}{2}}\le C.
	\end{align*}
	Substituting the estimates of $\mathcal{I}_1$, $\mathcal{I}_2$, $\mathcal{I}_3$ into \eqref{eq:theta=1/2}, we finish the proof of \eqref{eq:Holder}.
\end{proof}

The half-order fractional-domain estimate established in Lemma~\ref{lem:Holder} yields the following optimal temporal H\"older regularity, which is essential for the strong convergence analysis in Section~\ref{S4}.

\begin{theorem}\label{lem:opt-Hol}
	Let Assumption \ref{asp:aT} and $\chi^\wedge\in D({\mathcal{A}_\lambda^{1/2}})$ hold. Then for any $p\ge2$, there exists $C>0$ such that for any $0\le s\le t\le T$,
	\begin{equation}\label{eq:opt-Hol}
		\|u(t)-u(s)\|_{L^p(\Omega,L^2_{\beta\gamma}(\Gamma))}\le C(t-s)^{\frac{1}{2}}.
	\end{equation}
\end{theorem}
\begin{proof}
	Due to \eqref{eq:v-v} and \eqref{eq:Kt}, we obtain that for any $0\le s\le t\le T$ and $p\ge 2$,
	\begin{equation*}
		\E\big[\|v_\lambda(t)-v_\lambda(s)\|_{L^2_{\beta\gamma}(\Gamma)}^p\big]\le C\E\big[\|{\mathcal{A}_\lambda^{-1/2}}(S_\lambda(t-s)-I){\mathcal{A}_\lambda^{1/2}}v_\lambda(s)\|_{L^2_{\beta\gamma}(\Gamma)}^p\big]+C(t-s)^{\frac{p}{2}}.
	\end{equation*}
	Furthermore, applying Lemma  \ref{lem:Holder}, as well as  Lemma \ref{lem:smo} leads to
	\begin{equation*}
		\E\left[\|v_\lambda(t)-v_\lambda(s)\|_{L^2_{\beta\gamma}(\Gamma)}^p\right]\le C(t-s)^{\frac{p}{2}}.
	\end{equation*}
	Recalling that $u(t)=e^{\lambda t}v_\lambda(t)$, we use
	\begin{align*}
	 \|u(t)-u(s)\|_{L^p(\Omega,L^2_{\beta\gamma}(\Gamma))}
	 &\le e^{\lambda t}\|v_\lambda(t)-v_\lambda(s)\|_{L^p(\Omega,L^2_{\beta\gamma}(\Gamma))}\\
	 &\quad+|e^{\lambda t}-e^{\lambda s}|\,\|v_\lambda(s)\|_{L^p(\Omega,L^2_{\beta\gamma}(\Gamma))},
	\end{align*}
	together with \eqref{eq:L2}, and the proof of the temporal H\"older regularity estimate \eqref{eq:opt-Hol} is complete.
\end{proof}

The exponent $1/2$ in \eqref{eq:opt-Hol} is optimal in general.
Indeed, consider the special case $b=0$, $g\equiv1$, and
$\chi^\wedge=0$, and assume that the Nemytskii operator $G$ in
\eqref{eq:BG} associated with $g\equiv1$ is nonzero, i.e.,
$\|G\|_{\mathscr{L}_2(\mathcal U_0,L^2_{\beta\gamma}(\Gamma))}>0$.
The mild solution of \eqref{eq:intro} is then the Ornstein--Uhlenbeck process
$u(t)=\int_0^t S(t-r)G\,\ud W(r)$. By the It\^o isometry,
\begin{equation*}
\mathbb{E}\|u(t)-u(0)\|_{L^2_{\beta\gamma}(\Gamma)}^2
=\int_0^t
\|S(r)G\|_{\mathscr{L}_2(\mathcal U_0,L^2_{\beta\gamma}(\Gamma))}^2
\,\ud r\sim t\|G\|_{\mathscr{L}_2(\mathcal U_0,L^2_{\beta\gamma}(\Gamma))}^2,
\qquad t\downarrow0.
\end{equation*}
Here the asymptotic relation follows from the strong continuity of $S$
and the fact that $G$ is Hilbert--Schmidt. Hence, in general, no
temporal H\"older exponent larger than $1/2$ can hold.

\section{Strong convergence and asymptotic preservation}\label{S4}

As shown in section~\ref{S:discretelimit}, the remaining step in
establishing the AP property \eqref{eq:AP-intro} is to prove the strong
convergence of the limiting scheme \eqref{eq:Un-intro}. We first
state and prove this convergence result.
	\begin{theorem}\label{thm:main2}
	Suppose that Assumption  \ref{asp:aT} holds, and that $b$ and $g$ are globally Lipschitz continuous. Let
$\{U^n\}_{n=0}^N$ be the exponential Euler approximation
\eqref{eq:Un-intro} with step size $\tau=T/N$. Then the following
statements hold.

(a) If
$\chi^\wedge\in L^2_{\beta\gamma}(\Gamma)$, then, for every
$\varrho\in(0,\frac{1}{2})$ and $p\geq2$, there exists a constant $C_\varrho:=C(p,T,\chi^\wedge,b,g,\mu,\gamma,\varrho)>0$ such that for any $n=1,\ldots,N$,
\begin{equation}\label{eq:main}
\mathbb E\left[
\|u(t_n)-U^n\|_{L^2_{\beta\gamma}(\Gamma)}^p
\right]
\leq C_\varrho\tau^{p\varrho}.
\end{equation}

(b) If
$\chi^\wedge\in D({\mathcal A_\lambda^{1/2}})$ for some $\lambda>\kappa/8$, then
\eqref{eq:main} remains valid for $\varrho=\frac{1}{2}$.
\end{theorem}

\begin{proof}
We give a unified proof of both assertions. Let $\varrho_1=\varrho$ if $\varrho\in (0,\frac{1}{2})$; and $\varrho_1=0$ if $\varrho=\frac{1}{2}$.
 Recall that by Lemma \ref{lem:uHol1} and Theorem \ref{lem:opt-Hol},
\begin{equation}\label{eq:unified-Hol}
\|u(t)-u(s)\|_{L^p(\Omega,L^2_{\beta\gamma}(\Gamma))}
\leq C(p,\varrho,\chi^\wedge,T) s^{-\varrho_1}|t-s|^{\varrho},
\qquad 0< s\leq t\leq T,
\end{equation}
where $\varrho\in (0,\frac{1}{2})$ if $\chi^\wedge\in L^2_{\beta\gamma}(\Gamma)$, and $\varrho=\frac{1}{2}$ if $\chi^\wedge\in D({\mathcal A_\lambda^{1/2}})$.
For each $n=0,1,\ldots,N$, by denoting $V_\lambda^n=e^{-\lambda t_n}U^n$ and using \eqref{eq:Un-intro},
\begin{equation*}%
	V_\lambda^n=S_\lambda(\tau)V_\lambda^{n-1}
	+\tau S_\lambda(\tau)e^{-\lambda t_{n-1}}B(U^{n-1})
	+S_\lambda(\tau)e^{-\lambda t_{n-1}}G(U^{n-1})\delta W_n
\end{equation*}
for any $n=1,2,\ldots,N$. Note that $V_\lambda^n$ is the exponential Euler approximation of $v_\lambda(t_n)$ (see \eqref{eq:vlambda}). By iteration, we have 
\begin{equation*}%
	V_\lambda^n=S_\lambda(t_n)V_\lambda^{0}+\tau \sum_{j=0}^{n-1}S_\lambda(t_n-t_j)e^{-\lambda t_{j}}
	B(U^j)
	+ \sum_{j=0}^{n-1}S_\lambda(t_n-t_j)e^{-\lambda t_{j}}
	G(U^j)\delta W_{j+1}.
\end{equation*}
Hence, taking into account \eqref{eq:vlambda}, for
$n=1,\ldots,N$ we obtain
\begin{equation}\label{eq:v-vn}
v_\lambda(t_n)-V_\lambda^n
=\mathcal D_1+\mathcal D_2+\mathcal D_3
+\mathcal T_1+\mathcal T_2+\mathcal T_3,
\end{equation}
where deterministic error terms $\{\mathcal D_i\}_{i=1}^3$ and the stochastic error terms $\{\mathcal T_i\}_{i=1}^3$ are
\begin{align*}
\mathcal D_1
&:=
\sum_{j=0}^{n-1}
\int_{t_j}^{t_{j+1}}
\bigl(S_\lambda(t_n-s)-S_\lambda(t_n-t_j)\bigr)
e^{-\lambda s}B(u(s))\,\ud s,\\
\mathcal D_2
&:=
\sum_{j=0}^{n-1}
\int_{t_j}^{t_{j+1}}
S_\lambda(t_n-t_j)
\bigl(
e^{-\lambda s}B(u(s))
-e^{-\lambda t_j}B(u(t_j))
\bigr)\,\ud s,\\
\mathcal D_3
&:=
\sum_{j=0}^{n-1}
\int_{t_j}^{t_{j+1}}
S_\lambda(t_n-t_j)e^{-\lambda t_j}
\bigl(B(u(t_j))-B(U^j)\bigr)\,\ud s,\\
\mathcal T_1
&:=
\sum_{j=0}^{n-1}
\int_{t_j}^{t_{j+1}}
\bigl(S_\lambda(t_n-s)-S_\lambda(t_n-t_j)\bigr)
e^{-\lambda s}G(u(s))\,\ud W(s),\\
\mathcal T_2
&:=
\sum_{j=0}^{n-1}
\int_{t_j}^{t_{j+1}}
S_\lambda(t_n-t_j)
\bigl(
e^{-\lambda s}G(u(s))
-e^{-\lambda t_j}G(u(t_j))
\bigr)\,\ud W(s),\\
\mathcal T_3
&:=
\sum_{j=0}^{n-1}
\int_{t_j}^{t_{j+1}}
S_\lambda(t_n-t_j)e^{-\lambda t_j}
\bigl(G(u(t_j))-G(U^j)\bigr)\,\ud W(s).
\end{align*}
 Let $p\in[2,\infty)$ and denote $\mathscr{L}_2^0:=\mathscr{L}_2(\mathcal{U}_0,L_{\beta\gamma}^2(\Gamma))$. 
 We first estimate the error terms $\mathcal T_1$ and $\mathcal{D}_1$ due to the semigroup freezing. By Burkholder inequality and Lemma \ref{lem:smo}, 
		\begin{align*}
		&\E\big[\|\mathcal{T}_1\|_{L_{\beta\gamma}^2(\Gamma)}^p\big]
		\le C\E\Bigg[\bigg(\sum_{j=0}^{n-1}\int_{t_j}^{t_{j+1}} \|(S_\lambda(t_n-s)- S_\lambda(t_n-t_j))e^{-\lambda s}
		G(u(s))\|_{\mathscr{L}_2^0}^2\ud s\bigg)^{\frac{p}{2}}\Bigg]\\
			&\le C\E\Bigg[\bigg(\sum_{j=0}^{n-1}\int_{t_j}^{t_{j+1}} \|(I- S_\lambda(s-t_j))S_\lambda(t_n-s)
		G(u(t_n))\|_{\mathscr{L}_2^0}^2\ud s\bigg)^{\frac{p}{2}}\Bigg]\\
		&\quad+ C\E\Bigg[\bigg(\sum_{j=0}^{n-1}\int_{t_j}^{t_{j+1}} \|(I- S_\lambda(s-t_j))S_\lambda(t_n-s)
	(	G(u(t_n))-G(u(s)))\|_{\mathscr{L}_2^0}^2\ud s\bigg)^{\frac{p}{2}}\Bigg]\\
	&=:\mathcal{J}_1+\mathcal{J}_2.
		\end{align*}
		Invoking \eqref{eq:smo2}, \eqref{eq:smo1} (if $\varrho\in(0,\frac{1}{2})$), \eqref{eq:smooth} (if $\varrho=\frac{1}{2}$), and \eqref{eq:L2}, one has 
		\begin{align*}
		\mathcal{J}_1&\le C\E\Bigg[\bigg(\sum_{j=0}^{n-1}\int_{t_j}^{t_{j+1}} (s-t_j)^{2\varrho}\|{\mathcal{A}_\lambda^{\varrho}}S_\lambda(t_n-s)
		G(u(t_n))\|_{\mathscr{L}_2^0}^2\ud s\bigg)^{\frac{p}{2}}\Bigg]\\
		&\le C\tau^{p\varrho}\E\Bigg[\bigg(\int_{0}^{t_{n}} \|{\mathcal{A}_\lambda^{\varrho}}S_\lambda(t_n-s)
		G(u(t_n))\|_{\mathscr{L}_2^0}^2\ud s\bigg)^{\frac{p}{2}}\Bigg]\\
		&\le C\tau^{p\varrho}\E\left[\|G(u(t_n))\|_{\mathscr{L}_2^0}^p\right]\le C\tau^{p\varrho},
		\end{align*}
For completeness, the last step follows, when $0<\varrho<1/2$, from
		\[
		 \int_0^{t_n}\|\mathcal A_\lambda^\varrho S_\lambda(r)Q\|_{\mathscr L_2^0}^2\,\ud r
		 \le C\int_0^{t_n}r^{-2\varrho}\,\ud r\,\|Q\|_{\mathscr L_2^0}^2
		 \le C(T,\varrho)\|Q\|_{\mathscr L_2^0}^2,
		\]
		whereas for $\varrho=1/2$ it follows by applying \eqref{eq:smooth} to each column of $Q$ and summing.
		Here we have also used the linear growth of $G$ in \eqref{eq:Lin}. Moreover, according to Lemma \ref{lem:smo}, the Lipschitz continuity of $G$ in \eqref{eq:Lip}, the Minkowski inequality, and \eqref{eq:unified-Hol},
		\begin{align*}
			\mathcal{J}_2&\le C\E\Bigg[\bigg(\sum_{j=0}^{n-1}\int_{t_j}^{t_{j+1}} (s-t_j)^{2\varrho}(t_n-s)^{-2\varrho}\|	G(u(t_n))-G(u(s))\|_{\mathscr{L}_2^0}^2\ud s\bigg)^{\frac{p}{2}}\Bigg]\\
			&\le C\bigg(\sum_{j=0}^{n-1}\int_{t_j}^{t_{j+1}} (s-t_j)^{2\varrho}(t_n-s)^{-2\varrho}s^{-2\varrho_1}(t_n-s)^{2\varrho}\ud s\bigg)^{\frac{p}{2}}\le C_{\varrho}\tau^{p\varrho},
		\end{align*}
		where the last step used 
 $\varrho_1\in[0,\frac{1}{2})$. 
		Gathering the estimates of $\mathcal{J}_1$ and $\mathcal{J}_2$, we arrive at
			$\E[\|\mathcal{T}_1\|_{L_{\beta\gamma}^2(\Gamma)}^p]\le C\tau^{p\varrho}$.
		For the deterministic semigroup-freezing term, the same factorization and the linear growth of $B$ give the following estimate
		\begin{align*}
		 \|\mathcal D_1\|_{L^p(\Omega,L^2_{\beta\gamma}(\Gamma))}
		 &\le C\sum_{j=0}^{n-1}\int_{t_j}^{t_{j+1}}
		 (s-t_j)^\varrho(t_n-s)^{-\varrho}
		 \bigl(1+\|u(s)\|_{L^p(\Omega,L^2_{\beta\gamma}(\Gamma))}\bigr)\,\ud s\\
		 &\le C\tau^\varrho\int_0^{t_n}(t_n-s)^{-\varrho}\,\ud s
		 \le C_\varrho\tau^\varrho.
		\end{align*}

		Next, we estimate the error terms $\mathcal T_2$ and $\mathcal{D}_2$ due to the solution freezing.
		By the linear growth and Lipschitz continuity of $G$, we deduce that for any $s\in[t_j,t_{j+1}]$ and $j=0,1,\ldots,N-1$,
		\begin{equation*}
\|e^{-\lambda s}
G(u(s))-e^{-\lambda t_{j}}
G(u(t_j))\|_{\mathscr{L}_2^0}
\le C\tau(1+\|u(s)\|_{L_{\beta\gamma}^2(\Gamma)})+e^{-\lambda t_{j}}\|
u(s)-
u(t_j)\|_{L_{\beta\gamma}^2(\Gamma)}.
\end{equation*}
Then by \eqref{eq:L2} and \eqref{eq:unified-Hol}, we obtain that for any $s\in[t_j,t_{j+1}]$ and $j=1,\ldots,N-1$,
\begin{equation*}
	\|e^{-\lambda s}
	G(u(s))-e^{-\lambda t_{j}}
	G(u(t_j))\|_{L^{p}(\Omega,\mathscr{L}_2^0)}^2\le C\tau^2+C\tau^{2\varrho}t_j^{-2\varrho_1}\le C(1+t_j^{-2\varrho_1})\tau^{2\varrho}.
\end{equation*}
Hence, by the contractivity of $S_\lambda(\cdot)$, the Burkholder inequality, and the Minkowski inequality, we arrive at 
		\begin{align*}
			\E\big[\|\mathcal{T}_2\|_{L_{\beta\gamma}^2(\Gamma)}^p\big]&\le C\E\Bigg[\bigg(\sum_{j=0}^{n-1}\int_{t_j}^{t_{j+1}} \|e^{-\lambda s}
			G(u(s))-e^{-\lambda t_{j}}
			G(u(t_j))\|_{\mathscr{L}_2^0}^2\ud s\bigg)^{\frac{p}{2}}\Bigg]\\
			&\le C\bigg(\sum_{j=0}^{n-1}\int_{t_j}^{t_{j+1}} \|e^{-\lambda s}
			G(u(s))-e^{-\lambda t_{j}}
			G(u(t_j))\|_{L^{p}(\Omega,\mathscr{L}_2^0)}^2\ud s\bigg)^{\frac{p}{2}}\\
			&\le C\bigg(\tau +\tau^{2\varrho}\sum_{j=1}^{n-1}\int_{t_j}^{t_{j+1}} (1+t_j^{-2\varrho_1})\ud s\bigg)^{\frac{p}{2}}
			\le C_{\varrho}\tau^{p\varrho}.
\end{align*}
On the interval $[0,t_1]$, the moment bound \eqref{eq:L2} gives an $O(\tau)$ contribution inside the parentheses; since $0<\tau\leq1$ and $2\varrho\leq1$, this is $O(\tau^{2\varrho})$. For the remaining intervals, we used
$
\tau\sum_{j=1}^{N-1}t_j^{-2\varrho_1}\le C
$
for $2\varrho_1<1$.
For the deterministic term, the same decomposition, Minkowski's inequality, and the Lipschitz continuity of $B$ yield
\begin{align*}
 \|\mathcal D_2\|_{L^p(\Omega,L^2_{\beta\gamma}(\Gamma))}
 &\le C\tau
 +C\sum_{j=1}^{n-1}\int_{t_j}^{t_{j+1}}
 \bigl(\tau+t_j^{-\varrho_1}(s-t_j)^\varrho\bigr)\,\ud s\le C_\varrho\tau^\varrho.
\end{align*}

It remains to estimate the error terms $\mathcal T_3$ and $\mathcal D_3$.
Using the Burkholder inequality, the Minkowski inequality, the Lipschitz continuity of $G$, and the discrete H\"older inequality, we obtain
\begin{align*}
	\E\left[\|\mathcal{T}_3\|_{L_{\beta\gamma}^2(\Gamma)}^p\right]
	&\le C\E\Bigg[\bigg(\sum_{j=0}^{n-1}\int_{t_j}^{t_{j+1}} \|G(u(t_j))-G(U^j)\|_{\mathscr{L}_2^0}^2\ud s\bigg)^{\frac{p}{2}}\Bigg]\\
	&\le C\bigg(\sum_{j=0}^{n-1}\int_{t_j}^{t_{j+1}} \|G(u(t_j))-G(U^j)\|_{L^{p}(\Omega,\mathscr{L}_2^0)}^2\ud s\bigg)^{\frac{p}{2}}\\
	&\le C\sum_{j=0}^{n-1}\int_{t_j}^{t_{j+1}} \|u(t_j)-U^j\|_{L^{p}(\Omega,L_{\beta\gamma}^2(\Gamma))}^{p}\ud s.
\end{align*}
Similarly, by the Lipschitz continuity of $B$, we also have 
\begin{equation*}
	\E\left[\|\mathcal{D}_3\|_{L_{\beta\gamma}^2(\Gamma)}^p\right]
		\le C\sum_{j=0}^{n-1}\int_{t_j}^{t_{j+1}} \|u(t_j)-U^j\|_{L^{p}(\Omega,L_{\beta\gamma}^2(\Gamma))}^{p}\ud s.
\end{equation*}
Combining the estimates of $\mathcal{T}_i$ and $\mathcal D_i$ for $i=1,2,3$ with \eqref{eq:v-vn}, and since $u(t_n)-U^n=e^{\lambda t_n}(v_\lambda(t_n)-V_\lambda^n)$, we arrive at 
\begin{equation*}
\|u(t_n)-U^n\|_{L^{p}(\Omega,L_{\beta\gamma}^2(\Gamma))}^{p}\le C\tau^{p\varrho}+C\sum_{j=0}^{n-1}\int_{t_j}^{t_{j+1}} \|u(t_j)-U^j\|_{L^{p}(\Omega,L_{\beta\gamma}^2(\Gamma))}^{p}\ud s.
\end{equation*}
Finally, applying the discrete Gr\"onwall inequality, we complete the proof.
\end{proof}

Combining Theorem~\ref{thm:main2} with
Propositions~\ref{tho:faa} and~\ref{theo:MS-2D}
yields the AP property of the exponential Euler method \eqref{eq:EEMmult} in the
multiple-critical-point setting. This extends the AP result of  \cite{CS24}, where the analysis was restricted to the case of a Hamiltonian with a unique critical point.
\begin{theorem}\label{coro:AP}
Let $b$ and $g$ be continuously differentiable with bounded derivatives, and let Assumption \ref{Asp:gamma-S} and \eqref{eq:Tneq0} hold. Assume that
$\int_{\R^2}|\chi(x)|^2\sqrt{\vartheta(H(x))}\ud x<\infty$, $\nabla\chi\in \mathbb{H}_\gamma$, $\chi=(\chi^\wedge)^\vee$, and $\int_{\R^2}|\xi|^2\ud\mu<\infty$. Then the following limits commute:
 \begin{equation*}
\lim_{\epsilon\to 0}\lim_{N\to \infty}U_{\epsilon}^{N}=u(T)^\vee=\lim_{N\to \infty}\lim_{\epsilon\to 0}U_{\epsilon}^{N}\quad\text{in} ~L^2(\Omega,\mathbb{H}_\gamma),
\end{equation*}
where 
$U_{\epsilon}^{N}$ is the exponential Euler approximation of $u_\epsilon(T)$, given by \eqref{eq:EEMmult}.
\end{theorem}

\begin{proof}
Recall that $\gamma(z,k)=\vartheta(z)$ and that $\gamma$ is bounded.
Hence, the assumption
$
\int_{\mathbb R^2}
|\chi(x)|^2\sqrt{\vartheta(H(x))}\,\ud x<\infty
$
implies that $\chi\in\mathbb H_\gamma$, because
$\vartheta\leq \|\vartheta\|_\infty^{1/2}\sqrt\vartheta$. Moreover, by the
projection estimate \eqref{eq:contraction}, $\chi^\wedge\in L^2_{\beta\gamma}(\Gamma)$.
Thus, the initial-value condition in Theorem \ref{thm:main2}(a) is
satisfied. Furthermore, Assumption \ref{Asp:gamma-S} implies both
Assumption \ref{Asp:gamma} and Assumption \ref{asp:aT}. Therefore,
Proposition \ref{theo:MS-2D}, Proposition \ref{tho:faa}, and Theorem
\ref{thm:main2}(a) are all applicable.
For every fixed $\epsilon>0$, Proposition \ref{theo:MS-2D}(1) gives
\begin{align*}
\limsup_{N\to \infty}
\mathbb E\big[
\|U_\epsilon^N-u(T)^\vee\|_{\mathbb H_\gamma}^2
\big]&\leq
2\limsup_{N\to \infty}
\mathbb E\big[
\|U_\epsilon^N\!-\!u_\epsilon(T)\|_{\mathbb H_\gamma}^2
\big]
\!+\!
2\mathbb E\big[
\|u_\epsilon(T)\!-\!u(T)^\vee\|_{\mathbb H_\gamma}^2
\big]\\
&\leq
2\mathbb E\big[
\|u_\epsilon(T)-u(T)^\vee\|_{\mathbb H_\gamma}^2
\big].
\end{align*}
Letting $\epsilon\to0$ and applying Proposition \ref{tho:faa}, we
obtain
$
\lim_{\epsilon\to0}\limsup_{N\to \infty}
\mathbb E[
\|U_\epsilon^N-u(T)^\vee\|_{\mathbb H_\gamma}^2
]
=0.
$
This proves the first sequential limit.

Conversely, for a fixed $\tau>0$, Proposition
\ref{theo:MS-2D}(2) and the isometry in \eqref{eq:contraction} yield
\begin{align*}
&\limsup_{\epsilon\to0}
\mathbb E\left[
\|U_\epsilon^N-u(T)^\vee\|_{\mathbb H_\gamma}^2
\right]\\
&\quad\leq
2\limsup_{\epsilon\to0}
\mathbb E\left[
\|U_\epsilon^N-(U^N)^\vee\|_{\mathbb H_\gamma}^2
\right]
+
2\mathbb E\left[
\|U^N-u(T)\|_{L^2_{\beta\gamma}(\Gamma)}^2
\right]\\
&\quad=
2\mathbb E\left[
\|U^N-u(T)\|_{L^2_{\beta\gamma}(\Gamma)}^2
\right].
\end{align*}
Furthermore, Theorem
\ref{thm:main2}(a) gives
$
\lim_{N\to \infty}\limsup_{\epsilon\to0}
\mathbb E[
\|U_\epsilon^N-u(T)^\vee\|_{\mathbb H_\gamma}^2
]
=0,
$
which proves the second sequential limit. The proof is complete.
\end{proof}
\section{Numerical experiments}\label{S:NE}

We consider the Hamiltonian
\begin{equation}\label{H}
 H(\mathtt{x},\mathtt{y})=2\mathtt{y}^2+V(\mathtt{x}),
 \qquad (\mathtt{x},\mathtt{y})\in\mathbb R^2,
\end{equation}
where $V(\mathtt{x})=V_1(\mathtt{x})-V_1(-2)$ and
\[
 V_1(\mathtt{x})
 =\frac25\mathtt{x}^2+\frac45\mathtt{x}
 -\frac32\log(1+\mathtt{x}^2)-\frac{11}{10}\arctan\mathtt{x}
 +\frac{3\mathtt{x}-9}{10(1+\mathtt{x}^2)}.
\]
Unless otherwise stated, the test data are
\begin{equation}\label{eq:fv-common-test-data}
 b\equiv0,~ g(u)=\cos u,~
 \chi(x)=e^{-H(x)},~
 \gamma(z,k)=e^{-\sqrt{z+1}},~
 \Lambda(x)=\pi^{-1}e^{-|x|^2}.
\end{equation}

Fig.~\ref{Hamilton_level} displays the Hamiltonian surface for \eqref{H},
its level sets, and the associated {\color{black}metric graph}.
The minima $(-2,0)$ and $(1,0)$ and the saddle $(0,0)$ correspond to
$O_1$, $O_2$, and $O_3$, respectively. The associated {\color{black}metric graph} has two
bounded edges $I_1,I_2$ joining the minima to $O_3$ and one unbounded
edge $I_3$ extending from $O_3$ toward $O_\infty$; see
Fig.~\ref{Hamilton_level}.
The projected initial datum $\chi^\wedge(z,k)=e^{-z}$ belongs to
$D(\mathcal L)=D(\mathcal A_\lambda)$: it has the required edgewise
regularity and integrability, is continuous at the vertices, and
satisfies the Kirchhoff condition because
$\alpha_3(H(O_3))=\alpha_1(H(O_3))+\alpha_2(H(O_3))$.

We first construct a finite-volume graph discretization in
subsection~\ref{S:full-discrete}, then examine the fast-advection asymptotic property in subsection~\ref{subsec:ap-comparison} and the temporal
convergence predicted by Theorem~\ref{thm:main2} in
subsection~\ref{subsec:graph-temporal-convergence}.
Our theoretical analysis addresses only the temporal discretization,
whereas the numerical implementation additionally involves graph
truncation, spatial discretization, covariance discretization, and the
numerical generation of Gaussian increments.
Accordingly, the experiments should be interpreted as providing
qualitative numerical evidence for the AP property rather than a
fully discrete AP error estimate.
We refer to \cite{CKS26} for related error estimates for graph truncation
and spatial discretization of \eqref{eq:intro}.

\subsection{A truncated finite-volume spatial discretization}
\label{S:full-discrete}

We truncate $I_3$ at $z_\infty:=H(O_3)+L_\Gamma$, where $L_\Gamma>0$,
and denote the artificial vertex by $O_{\rm out}:=(z_\infty,3)$.
The truncated {\color{black}metric graph} $\Gamma_{L_\Gamma}$ has edges
\[
 I_{1,L_\Gamma}=I_1,\qquad I_{2,L_\Gamma}=I_2,\qquad
 I_{3,L_\Gamma}
 =\{(z,3)\in I_3:H(O_3)<z<z_\infty\}.
\]
We impose a homogeneous Dirichlet condition at $O_{\rm out}$ and
natural zero-flux conditions at $O_1,O_2$, where
$\alpha_k(H(O_k))=0$, $k=1,2$.
On each edge, $\partial_t v=\mathcal Lv$ takes the conservative form
\begin{equation}\label{eq:fv-conservative-form}
 \beta_k(z)\partial_t v(t,z,k)+\partial_zJ_k(t,z)=0,
 \qquad J_k(t,z)=-\frac12\alpha_k(z)\partial_zv(t,z,k).
\end{equation}

Partition $I_{k,L_\Gamma}$ into $N_z^k$ uniform cells
$C_i^k=(z_{i-1/2}^k,z_{i+1/2}^k)$, with centers $z_i^k$ and
width $\Delta z^k$. Let $w_i^k(t)$ approximate $v(t,z_i^k,k)$ and
set $N_h:=\sum_{k=1}^3N_z^k$.
Integrating \eqref{eq:fv-conservative-form} over a cell gives
\begin{equation}\label{eq:Cik}
 \frac{\ud}{\ud t}\int_{C_i^k}\beta_k(z)v(t,z,k)\,\ud z
 =J_k(t,z_{i-1/2}^k)-J_k(t,z_{i+1/2}^k).
\end{equation}
We approximate the cell integral in \eqref{eq:Cik} by $m_i^kw_i^k(t)$, where
$m_i^k:=\beta_k(z_i^k)\Delta z^k$.

\emph{Interior fluxes.}
Arithmetic averaging of $\alpha_k$ and a centered difference for
$\partial_zv$ give
\[
 J_k(\cdot,z_{i+1/2}^k)
 \approx c_{i+1/2}^k(w_i^k-w_{i+1}^k),
 \qquad
 c_{i+1/2}^k
 :=\frac{\alpha_k(z_i^k)+\alpha_k(z_{i+1}^k)}
          {4(z_{i+1}^k-z_i^k)},
\]
where the time arguments $t$ of the cell values $w^k_i$ and $w^k_{i+1}$ are suppressed.
The flux enters the two adjacent balances with opposite signs.
Thus, for $2\leq i\leq N_z^k-1$,
\begin{equation*}%
 m_i^k\frac{\ud w_i^k}{\ud t}
 =c_{i-1/2}^k(w_{i-1}^k-w_i^k)
  +c_{i+1/2}^k(w_{i+1}^k-w_i^k).
\end{equation*}

\emph{Vertex coupling and boundary fluxes.}
At $O_1,O_2$, the boundary fluxes vanish. At $O_3$, write
\[
 \begin{aligned}
 (z_{O,1},z_{O,2},z_{O,3})
   &:=(z_{N_z^1}^1,z_{N_z^2}^2,z_1^3),\qquad
 (w_{O,1},w_{O,2},w_{O,3})
   &:=(w_{N_z^1}^1,w_{N_z^2}^2,w_1^3).
 \end{aligned}
\]
The flux from the $k$th adjacent cell toward the common vertex value
$\widehat w_O$ is $d_k(w_{O,k}-\widehat w_O)$, with
$d_k:=\alpha_k(z_{O,k})/\Delta z^k$; the denominator accounts for
the half-cell distance to the vertex and the factor $1/2$ in the
flux. The discrete Kirchhoff condition and its solution are given by
\begin{equation}\label{eq:fv-discrete-kirchhoff}
 \sum_{k=1}^3d_k(w_{O,k}-\widehat w_O)=0,
 \qquad
 \widehat w_O=\frac{\sum_{k=1}^3d_kw_{O,k}}{\sum_{k=1}^3d_k}.
\end{equation}
Thus $\widehat w_O$ is determined by the cell values, with no
additional vertex degree of freedom. The adjacent cell equations are
\begin{equation*}%
 \begin{aligned}
 m_{N_z^k}^k\frac{\ud w_{N_z^k}^k}{\ud t}
 &=c_{N_z^k-1/2}^k(w_{N_z^k-1}^k-w_{N_z^k}^k)
   +d_k(\widehat w_O-w_{N_z^k}^k),
 &&k=1,2,\\
 m_1^3\frac{\ud w_1^3}{\ud t}
 &=d_3(\widehat w_O-w_1^3)+c_{3/2}^3(w_2^3-w_1^3).
 \end{aligned}
\end{equation*}
At $O_{\rm out}$, the outward flux is $d_\infty w_{N_z^3}^3$,
where $d_\infty:=\alpha_3(z_{N_z^3}^3)/\Delta z^3$, so
\begin{equation*}%
 m_{N_z^3}^3\frac{\ud w_{N_z^3}^3}{\ud t}
 =c_{N_z^3-1/2}^3(w_{N_z^3-1}^3-w_{N_z^3}^3)
  -d_\infty w_{N_z^3}^3.
\end{equation*}
These cell balances define the finite-volume matrix $L_h$ through
\[
 \frac{\ud V_h}{\ud t}=L_hV_h,
 \qquad
 V_h:=(w_i^k)_{k=1,2,3;\,i=1,\ldots,N_z^k}.
\]

\emph{Evaluation of the coefficients.}
For $(z,k)$ in the interior of $I_k$, let
$\mathtt{x}_-^k(z)<\mathtt{x}_+^k(z)$ be the roots of
$V(\mathtt{x})=z$ bounding the $\mathtt{x}$-projection of $C_k(z)$,
and set
\[
 Y_z(\mathtt{x}):=\sqrt{0.5(z-V(\mathtt{x}))},
 \qquad
 \mathtt{x}\in[\mathtt{x}_-^k(z),\mathtt{x}_+^k(z)].
\]
The orbit consists of $\mathtt{y}=\pm Y_z(\mathtt{x})$.
Using symmetry and the arc-length identity
$\ud l_{z,k}=|\nabla H(\mathtt{x},Y_z(\mathtt{x}))|
\ud\mathtt{x}/(4Y_z(\mathtt{x}))$, we obtain from \eqref{eq:AT} that
\[
 \begin{aligned}
 \beta_k(z)
 =2\int_{\mathtt{x}_-^k(z)}^{\mathtt{x}_+^k(z)}
       \frac{\ud\mathtt{x}}{4Y_z(\mathtt{x})},\qquad
 \alpha_k(z)
 =2\int_{\mathtt{x}_-^k(z)}^{\mathtt{x}_+^k(z)}
       \frac{|V'(\mathtt{x})|^2+16|Y_z(\mathtt{x})|^2}
            {4Y_z(\mathtt{x})}\,\ud\mathtt{x}.
 \end{aligned}
\]
To regularize the endpoint singularities, we use the change of variables
\[
 \mathtt{x}
 =\frac{\mathtt{x}_-^k(z)+\mathtt{x}_+^k(z)}2
 +\frac{\mathtt{x}_+^k(z)-\mathtt{x}_-^k(z)}2\cos q,
 \qquad q\in[0,\pi].
\]
The transformed integrals are evaluated at the energy-cell centers
using the midpoint rule with $M_q=512$ points.
Figure~\ref{fig:fv-alpha-beta} shows the resulting coefficients.
The sampled values of $\beta_k$ are strictly monotone on each edge,
in qualitative agreement with \eqref{eq:Tneq0}, although these
finite samples do not establish $\beta_k'(z)\neq0$ throughout an edge.

\begin{figure}[htbp]
  \centering
  \begin{minipage}[t]{0.48\textwidth}
    \centering
    \includegraphics[width=0.94\linewidth,height=4.8cm]
    {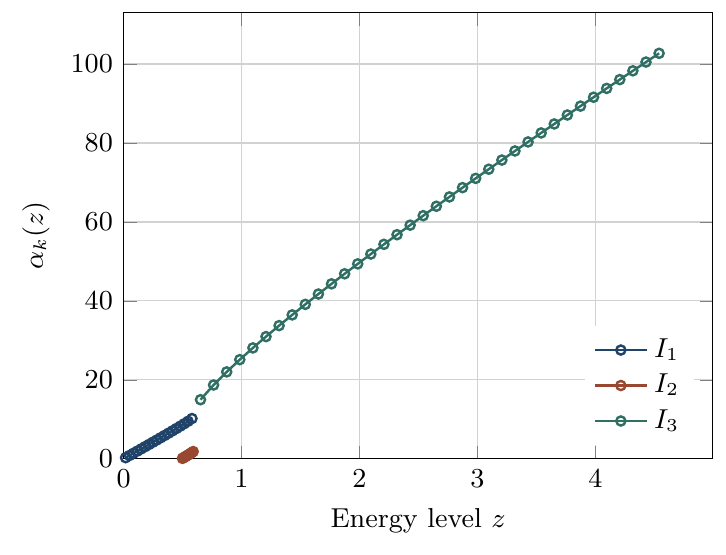}
  \end{minipage}
  \hfill
  \begin{minipage}[t]{0.48\textwidth}
    \centering
    \includegraphics[width=0.94\linewidth,height=4.8cm]
    {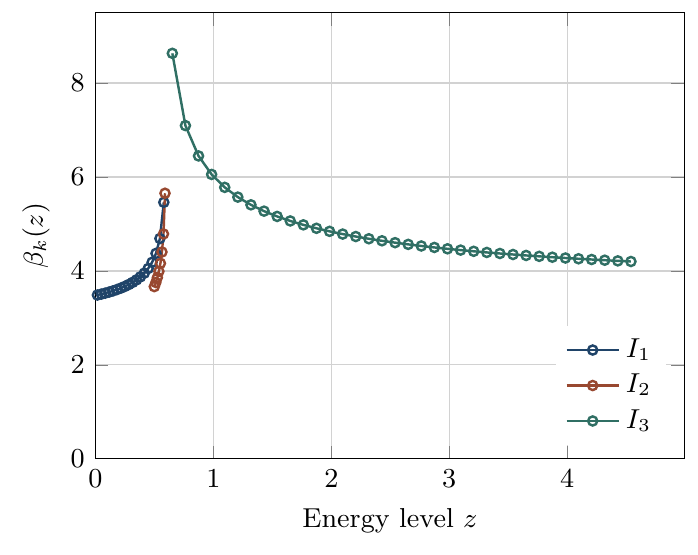}
  \end{minipage}
  \caption{Level-set coefficients at the energy-cell centers for
  $(N_z^1,N_z^2,N_z^3)=(18,8,36)$ and $L_\Gamma=4$, computed by
  midpoint quadrature with $M_q=512$ points.
  Left: $\alpha_k(z_i^k)$. Right: $\beta_k(z_i^k)$.}
  \label{fig:fv-alpha-beta}
\end{figure}

\subsection{Asymptotic property}
\label{subsec:ap-comparison}

We compare the exponential Euler approximation of
\eqref{eq:SRDA-intro} with the lifted approximation of the graph
equation \eqref{eq:intro}, using coupled noise increments.
For the physical equation \eqref{eq:SRDA-intro}, we use the centered Cartesian
finite-difference discretization of \cite[Appendix~B.1]{CS24} on
\[
 \Omega_{L_\Gamma}
 :=\{(\mathtt{x},\mathtt{y})\in\mathbb R^2:
       H(\mathtt{x},\mathtt{y})<H(O_3)+L_\Gamma\}.
\]
Let $\mathcal X_h=\{x_i\}_{i=1}^{N_x}$ be the Cartesian grid points
of spacing $h$ in $\Omega_{L_\Gamma}$, with grid functions extended
by zero outside $\mathcal X_h$. For
$p=(\mathtt{x}_p,\mathtt{y}_p)\in\mathcal X_h$, set
\[
 \begin{aligned}
 (D_hU)_p
 &:=\frac12\sum_{\ell=1}^2
       \frac{U_{p+he_\ell}-2U_p+U_{p-he_\ell}}{h^2},\\
 (B_hU)_p
 &:=-4\mathtt{y}_p\frac{U_{p+he_1}-U_{p-he_1}}{2h}
    +V'(\mathtt{x}_p)\frac{U_{p+he_2}-U_{p-he_2}}{2h},
 \end{aligned}
\]
where $e_1,e_2$ are the standard coordinate vectors.
The discrete generator is $A_{\epsilon,h}:=D_h+\epsilon^{-1}B_h$.
Using the data in \eqref{eq:fv-common-test-data}, the two schemes,
with $\tau=T/N$, are given by 
\begin{align}
 U_{\epsilon,h}^{n+1}
 &=\exp(\tau A_{\epsilon,h})
   \left[U_{\epsilon,h}^{n}
   +g(U_{\epsilon,h}^{n})\odot\Delta W_{h,n}^{x}\right],
 &U_{\epsilon,h}^{0}&=u_{0,h},
 \label{eq:sec53-physical-EE}\\
 V_h^{n+1}
 &=\exp(\tau L_h)
   \left[V_h^{n}+g(V_h^{n})\odot\Delta W_{h,n}^{\Gamma}\right],
 &V_h^{0}&=v_{0,h},
 \label{eq:sec53-graph-EE}
\end{align}
for $n=0,\ldots,N-1$, where $g$ and $\odot$ act componentwise,
$u_{0,h}(x_i)=e^{-H(x_i)}$, and $v_{0,h,i}^k=e^{-z_i^k}$.
The physical and graph noise increments in \eqref{eq:sec53-physical-EE} and \eqref{eq:sec53-graph-EE} must represent the same Gaussian
field; independent samples would introduce a discrepancy unrelated
to the fast-advection limit.
Enumerate the graph cells by $C_a$, $a=1,\ldots,N_h$, with centers
$(z_a,k_a)$. Let $\xi_a$ solve
$\dot\xi_a=\nabla^\perp H(\xi_a)$ on $C_{k_a}(z_a)$, with period
$\beta_{k_a}(z_a)$. Choose the orbit points and weights as
\[
 s_{a,m}=\frac{(m-\tfrac12)\beta_{k_a}(z_a)}{N_{\rm orb}},
 \qquad
 x_{a,m}=\xi_a(s_{a,m}),
 \qquad
 \omega_{a,m}=\frac1{N_{\rm orb}},
\]
for $m=1,\ldots,N_{\rm orb}$. These are midpoint quadrature points
in Hamiltonian time, computed by inverting the cumulative travel-time
map. The corresponding covariance blocks are
 \begin{gather*}
 (Q_h^{xx})_{ij}
 :=\Lambda(x_i-x_j),\qquad
 (Q_h^{\Gamma\Gamma})_{ab}
 :=\sum_{m,\ell=1}^{N_{\rm orb}}
       \omega_{a,m}\omega_{b,\ell}\Lambda(x_{a,m}-x_{b,\ell}),\\
 (Q_h^{x\Gamma})_{ia}
 :=\sum_{m=1}^{N_{\rm orb}}
       \omega_{a,m}\Lambda(x_i-x_{a,m}),
 \qquad Q_h^{\Gamma x}:=(Q_h^{x\Gamma})^\top.
 \end{gather*}
We factor their joint covariance by pivoted Cholesky:
\[
 Q_h^{\rm joint}
 :=\begin{pmatrix}
 Q_h^{xx}&Q_h^{x\Gamma}\\
 Q_h^{\Gamma x}&Q_h^{\Gamma\Gamma}
 \end{pmatrix}
 \approx FF^\top,
 \qquad F=\begin{pmatrix}F_x\\F_\Gamma\end{pmatrix},
\]
where $F_x\in\mathbb R^{N_x\times r}$ and
$F_\Gamma\in\mathbb R^{N_h\times r}$, with stopping criterion
$\max_i(Q_h^{\rm joint}-FF^\top)_{ii}\leq10^{-7}$.
At each time step, a common vector
$\Delta\beta_n\sim\mathcal N(0,\tau I_r)$ generates
\[
 \Delta W_{h,n}^{x}=F_x\Delta\beta_n,
 \qquad
 \Delta W_{h,n}^{\Gamma}=F_\Gamma\Delta\beta_n.
\]
Thus both schemes \eqref{eq:sec53-physical-EE} and \eqref{eq:sec53-graph-EE} are driven by coupled realizations of the same
approximate Gaussian field.

We next define a linear lifting operator $J_h$ from graph vectors to
Cartesian grid vectors. For a graph vector $v_h$ and each
$p\in\mathcal X_h$, $(J_hv_h)_p$ is obtained by evaluating the
piecewise linear interpolant of $v_h$ along the corresponding graph
edge at the energy coordinate $H(p)$.
At $O_3$, the interpolant uses the common vertex value from
\eqref{eq:fv-discrete-kirchhoff}; at $O_{\rm out}$, it uses zero.
With the weighted norm
$
 \|w\|_{\mathcal X_h}^2
 :=h^2\sum_{x_i\in\mathcal X_h}
       \gamma(H(x_i))|w(x_i)|^2,
$
we estimate the root-mean-square comparison error using $P$ Monte Carlo samples:
\[
\mathcal E_{h,L_\Gamma}^{(P)}(\epsilon)
:=
\left(
\frac{1}{P}\sum_{p=1}^{P}
\left|
U_{\epsilon,h}^{N}(\omega_p)
-J_hV_h^{N}(\omega_p)
\right|_{\mathcal X_h}^{2}
\right)^{1/2}.
\]

Figure~\ref{fig:ap-two-mesh} compares $h=0.1$ and $h=0.05$ with
$\tau=2^{-10}$. The graph resolution is held fixed, and the remaining
parameters are
\[
 L_\Gamma=4,\qquad T=0.25,\qquad
 (N_z^1,N_z^2,N_z^3,N_{\rm orb})=(72,32,144,16),\qquad P=100.
\]

\begin{figure}[htbp]
  \centering
  \includegraphics[width=0.6\textwidth]
  {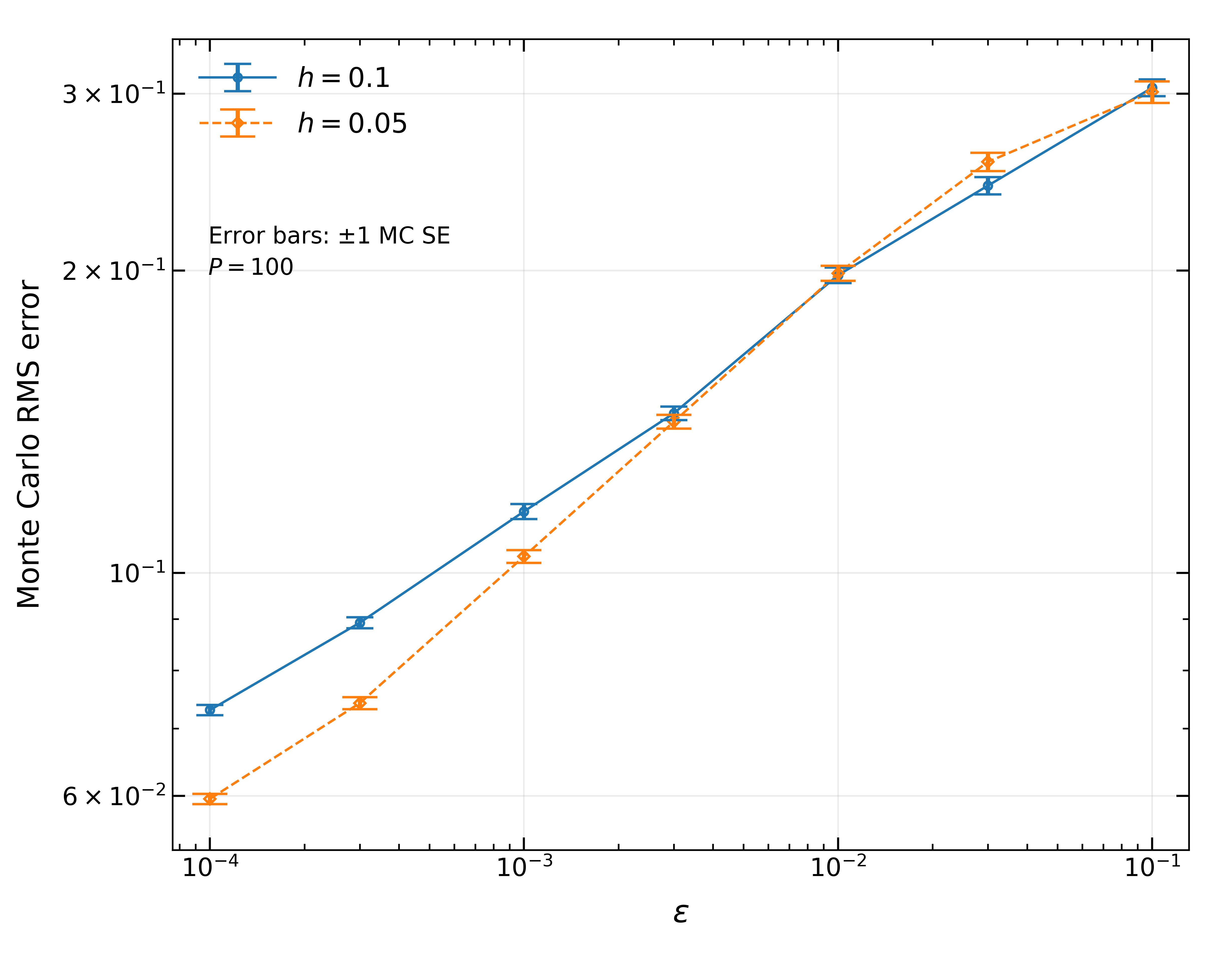}
  \caption{Monte Carlo comparison error
  $\mathcal E_{h,L_\Gamma}^{(P)}(\epsilon)$ for
  $h=0.1$ and $h=0.05$.
  Error bars indicate one Monte Carlo standard error on either
  side of each RMS estimate.}
  \label{fig:ap-two-mesh}
\end{figure}

On both Cartesian meshes, the comparison error  $\mathcal E_{h,L_\Gamma}^{(P)}(\epsilon)$ decreases as
$\epsilon$ decreases over the tested range,
in qualitative agreement with the time-discrete fast-advection limit.
At $\epsilon=10^{-4}$, refining $h$ from $0.1$ to $0.05$ reduces
the estimated error from $7.304\times10^{-2}$ to
$5.959\times10^{-2}$. The comparison also involves graph truncation,
spatial and covariance discretization, and the lifting operator $J_h$.
At fixed spatial resolution, these additional approximation errors
may dominate as $\epsilon$ decreases. The construction and analysis
of spatial discretizations that preserve the AP property are left
for future work.

\subsection{Temporal convergence for the limiting graph equation}
\label{subsec:graph-temporal-convergence}
We examine temporal convergence using the graph scheme
\eqref{eq:sec53-graph-EE} with the data in
\eqref{eq:fv-common-test-data}.
In the implementation, the covariance of the projected Wiener
process is approximated by orbit quadrature followed by spectral
truncation. Specifically, let $Q_h^{\Gamma\Gamma}$ be the discrete
orbit-averaged covariance matrix constructed in
Subsection~\ref{subsec:ap-comparison}, and let
$(\lambda_{\Gamma,j},q_{\Gamma,j})_{j=1}^{N_h}$ be its eigenpairs,
with orthonormal eigenvectors and eigenvalues ordered
nonincreasingly. We define
\[
 B_{\Gamma,h}
 :=\bigl(
 \sqrt{\lambda_{\Gamma,1}}q_{\Gamma,1},\ldots,
 \sqrt{\lambda_{\Gamma,r}}q_{\Gamma,r}
 \bigr),
 \qquad
 \Delta W_{h,n}^{\Gamma}=B_{\Gamma,h}\Delta B_n,
\]
where $\Delta B_n\sim\mathcal N(0,\tau I_r)$ are independent over
time steps and $r$ is the smallest rank satisfying
$
 \sum_{j>r}\lambda_{\Gamma,j}
 \leq10^{-12}\sum_{j=1}^{N_h}\lambda_{\Gamma,j}.
$
We use $T=1$, $L_\Gamma=8$, and $P=600$ sample paths, with two
spatial and orbit-quadrature resolutions:
\[
 (N_z^1,N_z^2,N_z^3,N_{\rm orb})=(12,5,48,8),
 \qquad
 (N_z^1,N_z^2,N_z^3,N_{\rm orb})=(24,10,96,16).
\]
The corresponding widths on the truncated third edge are
$\Delta z^3=1/6$ and $1/12$, the covariance ranks are $29$ and $31$,
and the relative Frobenius residuals are $5.57\times10^{-13}$ and
$3.05\times10^{-13}$, respectively.
Within each resolution, the spatial mesh and covariance factor
are held fixed while the time step is refined:
\[
 \tau_\ell=2^{-\ell},
 \qquad \ell=10,\ldots,15,
 \qquad N_\ell=T/\tau_\ell.
\]
For each sample path, increments are generated on the finest grid
$\tau_{15}$ and summed to obtain the coarser-grid increments, so
all temporal resolutions share the same Wiener realization.
Let $V_{h,\ell}^n$ denote the approximation with step $\tau_\ell$.
Using the weighted graph norm
\(
 \|v_h\|_{\mathrm{time},\Gamma,h}^{2}
 :=\sum_{k=1}^{3}\sum_{i=1}^{N_z^k}
 \beta_k(z_i^k)\Delta z^k\gamma(z_i^k,k)|v_i^k|^2,
\)
we measure the successive-level Cauchy errors by
\[
 \mathcal E_{\mathrm{time}}(\tau_\ell)
 :=\left(
 \frac1P\sum_{p=1}^{P}
 \left\|
 V_{h,\ell}^{N_\ell}(\omega_p)
 -V_{h,\ell+1}^{N_{\ell+1}}(\omega_p)
 \right\|_{\mathrm{time},\Gamma,h}^{2}
 \right)^{1/2},
 \qquad \ell=10,\ldots,14.
\]
The local empirical orders are
\(
 \operatorname{ord}_\ell
 :=\log_2(\mathcal E_{\mathrm{time}}(\tau_{\ell-1})/
      \mathcal E_{\mathrm{time}}(\tau_\ell))\) for $\ell=11,\ldots,14.$
Table~\ref{tab:fv-time-L8} shows decreasing Cauchy errors for both
resolutions. The fitted log--log slopes over $\ell=12,13,14$ are
$0.576$ and $0.587$, respectively, consistent with the strong
order-$1/2$ bound in Theorem~\ref{thm:main2}.

\begin{table}[htbp]
 \centering
 \small
 \renewcommand{\arraystretch}{1.08}
 \caption{Temporal Cauchy errors (Monte Carlo standard errors in
 parentheses) and local orders. Column groups specify
 $(N_z^1,N_z^2,N_z^3;N_{\rm orb})$.}
 \label{tab:fv-time-L8}
 \begin{tabular}{c|cc|cc}
 \hline
 & \multicolumn{2}{c|}{$(12,5,48;8)$}
 & \multicolumn{2}{c}{$(24,10,96;16)$}\\
 $\tau_\ell$
 & $\mathcal E_{\mathrm{time}}(\tau_\ell)$
 & $\operatorname{ord}_\ell$
 & $\mathcal E_{\mathrm{time}}(\tau_\ell)$
 & $\operatorname{ord}_\ell$\\
 \hline
 $2^{-10}$
 & $7.270\times10^{-4}\ (1.199\times10^{-5})$ & --
 & $6.962\times10^{-4}\ (1.110\times10^{-5})$ & --\\
 $2^{-11}$
 & $4.483\times10^{-4}\ (7.380\times10^{-6})$ & $0.697$
 & $4.504\times10^{-4}\ (7.980\times10^{-6})$ & $0.628$\\
 $2^{-12}$
 & $2.908\times10^{-4}\ (5.827\times10^{-6})$ & $0.625$
 & $3.019\times10^{-4}\ (6.439\times10^{-6})$ & $0.577$\\
 $2^{-13}$
 & $1.927\times10^{-4}\ (4.413\times10^{-6})$ & $0.593$
 & $2.025\times10^{-4}\ (4.140\times10^{-6})$ & $0.576$\\
 $2^{-14}$
 & $1.308\times10^{-4}\ (3.313\times10^{-6})$ & $0.559$
 & $1.339\times10^{-4}\ (3.597\times10^{-6})$ & $0.597$\\
 \hline
 \end{tabular}
\end{table}

\section{Conclusion}\label{S:conclusion}
In this work, we established an asymptotic-preserving result for the exponential
Euler method applied to stochastic reaction--diffusion--advection equations
with fast Hamiltonian transport and multiple critical points. A key
ingredient is the strong temporal convergence analysis of the limiting
scheme on the associated noncompact {\color{black}multi-edge metric graph}. Numerical experiments further illustrate the
fast-advection limit and the predicted temporal convergence behavior.
A natural direction for future work is the construction and rigorous
analysis of fully discrete AP schemes. In particular, it would be of
interest to quantify the effects of graph truncation, spatial
discretization in both the physical and graph spaces, and covariance
approximation uniformly with respect to the fast-advection parameter.

 \section*{Acknowledgments}

This work is supported by MOST National Key R\&D Program No. 2024YFA1015900, the Hong Kong Research Grants Council GRF grants 15301025 and 15302823, NSFC/RGC Joint Research Scheme N$\_$PolyU5141/24, NSFC grant 12522119, NSFC grant 12301526, internal funds (P0041274, P0045336) from the Hong Kong Polytechnic University, and the
CAS AMSS-PolyU Joint Laboratory of Applied Mathematics. We would like to thank Prof. Mihály Kovács (Chalmers University of Technology) for his valuable suggestions and helpful discussions, and Prof. Sandra Cerrai (University of Maryland) for her insightful lectures.

\bibliographystyle{siamplain}
\begingroup
\let\RevisionOriginalBibitem\bibitem
\renewcommand{\bibitem}[1]{\color{black}\RevisionOriginalBibitem{#1}%
  \def\RevisionCurrentKey{#1}\def\RevisionNewKey{V10}%
  \ifx\RevisionCurrentKey\RevisionNewKey\color{reviewred}\fi}
\bibliography{references}
\endgroup
\end{document}